\documentclass[dvipsnames]{aims}
\usepackage{amsmath,amssymb}
\usepackage{paralist}
\usepackage{graphics} 
\graphicspath{{figures/}}
\usepackage{epsfig} 
\usepackage{graphicx}  
\usepackage{epstopdf} 
\usepackage[colorlinks=true]{hyperref}
\hypersetup{urlcolor=blue, citecolor=red}

\def\currentvolume{X}
\def\currentissue{X}
\def\currentyear{200X}
\def\currentmonth{XX}
\def\ppages{X--XX}

\newtheorem{theorem}{Theorem}[section]

\newtheorem{lemma}[theorem]{Lemma}
\newtheorem{proposition}{Proposition}

\theoremstyle{definition}
\newtheorem{definition}[theorem]{Definition}
\newtheorem{remark}{Remark}

\usepackage{tikz}
\usepackage{tikz-cd}
\usetikzlibrary{arrows,trees,decorations,external,positioning,backgrounds}

\usepackage{bm}

\usepackage{mathtools}
\usepackage{float}

\title[Accessibility Properties of Abnormal Geodesics]{Accessibility
  Properties of Abnormal Geodesics in Optimal Control 
Illustrated by two case studies.} 

\keywords{Time minimal control for planar systems, 
  abnormal geodesics,  
  regularity of the value function,
  Zermelo navigation problems,
Chemical reaction networks.}

\author[Bernard Bonnard, Jérémy Rouot and Boris Wembe]{}
\email{bernard.bonnard@u-bourgogne.fr}
\email{jeremy.rouot@univ-brest.fr}
\email{boris.wembe@irit.fr}

\subjclass{Primary: 49K15, 49L99, 53C60, 58K50.}

\thanks{$^*$ Corresponding author: Jérémy Rouot.}

\def\dd{\mathrm d}

\begin{document}
\maketitle

\centerline{\scshape Bernard Bonnard}
\medskip
{\footnotesize
  \centerline{Institut de Mathématiques de Bourgogne, UMR 5584, Dijon}
  \centerline{INRIA, McTAO Team, Sophia Antipolis, France}
} 

\medskip

\centerline{\scshape Jérémy Rouot$^*$}
\medskip
{\footnotesize
  \centerline{Laboratoire de Mathématiques de Bretagne Atlantique, Brest, France}
}

\medskip

\centerline{\scshape Boris Wembe}
\medskip
{\footnotesize
  \centerline{Université Paul Sabatier, IRIT, Toulouse, France}
}

\bigskip

\centerline{(Communicated by the associate editor name)}

\begin{abstract}
  In this article, we use two case studies from
  geometry and optimal control of chemical network
  to analyze the relation between abnormal geodesics 
  in time optimal control, accessibility properties 
  and regularity of the time minimal value function.
\end{abstract}

\maketitle

\section*{Introduction}

In this article, one considers the time minimal control problem
for a smooth system of the form 
$\frac{\dd q}{\dd t}=f(q,u)$, where $q\in M$ is an open subset 
of $\mathbb{R}^n$ and the set of admissible control is the set $\mathcal{U}$
of bounded measurable mapping $u(\cdot)$ valued in a control domain $U$, where $U$
is a two-dimensional manifold of $\mathbb{R}^2$ with boundary.
According to the Maximum Principle \cite{pontryagin1962}, time minimal
solutions are extremal curves satisfying the constrained Hamiltonian 
equation
\begin{equation}
  \begin{aligned}
    &\dot q = \frac{\partial H}{\partial p},\quad 
    \dot p = -\frac{\partial H}{\partial q},
    \\
    &H(q,p,u) = \bm{M}(q,p),
  \end{aligned}
  \label{eq:pmp}
\end{equation}
where $H(q,p,u)=p\cdot F(q,u)$ is the pseudo (or non maximized) 
Hamiltonian, while $\bm{M}(q,p)= \max_{v\in U} H(q,p,u)$ is the 
true ({\it maximized}) Hamiltonian. A projection of an extremal 
curve $z=(q,p)$ on the $q$-space is called a geodesic.

Moreover since $\bm{M}$ is constant along an extremal curve and linear with respect
to $p$, the extremal can be either exceptional ({\it abnormal}) if $\bm{M}=0$
or non exceptional if $\bm{M}\neq 0$.
To refine this classification, an extremal subarc can be either regular 
if the control belongs to the boundary of $U$ or singular if it belongs
to the interior and satisfies the condition
$\frac{\partial H}{\partial u}=0$.

Taking $q(0)=q_0$ the accessibility set $A(q_0, t_f)$ in time $t_f$ is the 
set $\cup_{u(\cdot)\in \mathcal{U}}$ $q(t_f,x_0,u)$, 
where $t\mapsto q(\cdot,q_0,u)$ denotes the solution of the system, with
$q(0)=q_0$ and clearly since the time minimal trajectories belongs to the 
boundary of the accessibility set, the Maximum Principle is a 
parameterization of this boundary.
Since this set can have some pathologies \cite[p. 174]{kupka1980}, the 
analysis of the extremal dynamics is rather intricate.
The same holds for the time minimal value function 
$T(x_0,x_1)=\{\min\ t;\ x_1 = x(t,x_0,u),\ u(\cdot)\in \mathcal{U}\}$
even in the geodesically complete case.

This provides our geometric framework and to go further in our analysis 
we shall consider two case studies, each can be taken as a common 
thread in our analysis.

The first problem is one founding example of calculus of variations
and was set originally in 1931 by Zermelo and presented in details
by Carathéodory \cite{Zermelo1931,caratheodory1965}, in particular in the case of
linear wind for which we shall refer as the historical case along this paper.
The problem is a ship navigating on a river with a current and aiming to
reach the opposite shore.

Introducing the coordinates $(x,y)$ on the Euclidean space, where the 
current is given by $\mu(y) \frac{\partial}{\partial x}$ assuming only 
dependent upon the distance $y$ to the shore and the set of admissible 
direction is 
\[\dot q \in F_0(q) + U, \quad U=S^1.\]
This leads to study the time minimal control problem for the system
\[
  \dot {\tilde q} = F_0({\tilde q}) + \sum_{i=1}^2 u_i\, F_i({\tilde q}),
\]
where $F_1,F_2$ is an orthonormal frame for the Euclidean metric
denoted $g=\dd x^2 + \dd y^2$.
From a pure geometric point of view, one can generalize to a Zermelo
navigation problem on a surface of revolution $M\subset \mathbb{R}^3$,
endowed with the induced Riemannian metric, with parallel current and
represented as a triplet $(M,F_0,g)$.
The geodesics can be analyzed up to the action of the pseudo-group $G$ of 
smooth local change of coordinates.

The domain of navigation can be split into two subdomains :
\begin{itemize}
  \item strong current domain $\|F_0\|_g>1$ 
  \item  weak current domain : $\|F_0\|_g<1$
\end{itemize}
separated by the transitional case called moderate, with $\|F_0\|_g=1$.
In the historical problem with linear current $F_0=y\frac{\partial}{\partial x}$
it is given by $|y|=1$.

Introducing the pseudo-Hamiltonian, one gets 
\begin{equation}
  \begin{aligned}
    H(q,p,u) = H_0(q,p) + \sum_{i=1}^2 u_i H_i(q,p),
  \end{aligned}
\end{equation}
where $H_i(z) = p\cdot F_i(q),\ i=1,2$ are the Hamiltonian lifts.
In the historical problem we introduce the heading angle of the ship $\alpha$ 
and is the parameterization of the control since $F_1,F_2$ form a frame,
extremals are regular since $|u|=1$ belongs to the boundary of $U$:
$|u|\le 1$ and the Maximum Principle leads to compute the true
Hamiltonian : 
\[
  H(z) = H_0(z) + \sqrt{H_1^2(z) + H_2^2(z)}.
\]

But an equivalent point of view already introduced in the historical example
is to take as accessory control $v= \dot \alpha$, derivative of the heading angle.

Such transformation in relation with Goh transformation in optimal control 
will be called the Carathéodory-Zermelo-Goh transformation. Our study
boils down to time minimal control for the single-input control system 
\[
  \dot{\tilde q} = X(\tilde q)+ v\, Y(\tilde q), \quad v \in \mathbb{R},
\]
$X=F_0(q) + \cos \alpha F_1(q) + \sin \alpha F_2(q)$
and $Y=\frac{\partial}{\partial \alpha}$, $\tilde q =(q, \alpha)$.
Using this rewriting, the maximization condition gives us : 
$\frac{\partial H}{\partial \alpha}=0$ and hence extremals become
singular.
Moreover $p=(p_x, p_y)$ can be lifted into $\tilde p = (p, p_\alpha)$
and the Maximum Principle leads to the constraints:
\begin{align*}
  \tilde p(t)\cdot Y(\tilde q(t)) = \tilde p(t)\cdot [Y,X](\tilde q(t))=0,
\end{align*}
where 
$[Y,X](q)= \frac{\partial Y}{\partial q}(q)X(q)-\frac{\partial X}{\partial q}(q)Y(q)$
is the Lie bracket.

Abnormal geodesics are such that $H(q,p,u)=0$ and they were called limit curves 
in the historical example.
Their geometric interpretation is clear: they exist only in the strong current domain
and they are the limit curves of the cone of admissible directions.

One first objective of this article is to make a complete analysis of such abnormal
geodesics for the $2d$-Zermelo navigation problem, relating accessibility 
optimality to regularity of the value function.
It completes the series of results presented in \cite{bonnard2003}
describing the relations between singular trajectories in optimal control and 
feedback invariants.
They can be applied to more general problems, where the control is valued in a 
$2$-dimensional manifold with smooth boundary.

The main point is to analyze in this context optimality properties of geodesics 
which are non immersed curves using the techniques from singularity theory :
computing semi-normal forms and invariants in optimal control, 
see \cite{martinet1982} as a general reference for similar study for 
singularities of mappings, which goes back to the earliest work of Whitney 
\cite{whitney1955}.

The second case of interest consists the control of chemical reactions networks
like the McKeithan network:
$
\tikz[baseline={(a.base)},node distance=6mm] {
  \node (a) {T+M};
  \node[right=of a] (b) {A};
  \node[right=of b] (c) {B};
  \draw (a.base east) edge[left,->] node[above] {} (b.base west) (b.base east) edge[->] node[above] {} (c.base west);
  \draw (a) edge[bend right=40,<-] node[below right] {} (b);
  \draw (a.south) edge[bend right=60,<-] node[below right] {} (c);
}
$,
whose aim is to maximize the production of one species, e.g. $A$.
Assuming that the kinetics with respect to the temperature is
described by the Arrhenius law the dynamics can be modelled using 
the graph of reactions.
Taking the first derivative of the temperature as the control, which again
consists of a Goh transformation, the problem can be 
transformed into a time minimal control for a single-input affine 
control system.
The optimization problem can be transformed into a time minimal control problem,
where the terminal manifold $N$ is given by fixing at the final time a desired 
concentration of a chosen species, since both problems share the same geodesics.
A lot of preliminary work, see \cite{bonnard1997}, was done to analyze this problem.
In this article we shall concentrate to the so-called abnormal ({\it exceptional})
case where the geodesics are tangent to the terminal manifold $N$.
We complete the results from \cite{bonnard2019} to analyze this case.
Note that this problem can be set in the same frame than the Zermelo navigation
problem, where a barrier assimilated to a terminal manifold of codimension
one is given by $\|F_0\|_g=1$.
Again the geometric frame related to singularity theory was already in 
the earliest reference: construction of semi-normal form and concept of unfolding
in control related to the codimension of the singularities.
Recent progress of formal languages are used to describe algorithms to handle 
complicated computation, in particular in relation with the (reversible)
McKeithan network, to complete previous works justified by the network 
$A\rightarrow B \rightarrow C$.

For both case studies, the abnormal case is related to regularity properties 
of the value function, and discontinuity of the value function is analyzed in 
relation  with classification of accessibility properties.

This article is organized in three sections. In Section 1, we recall 
general results for time optimality for single-input affine
control systems, see \cite{bonnard2003}
as a general reference. Singular trajectories are introduced in relation
with feedback invariants and classified with respect to their optimality 
status.
Section 2 analyzes cusp singularities of abnormal geodesics in Zermelo 
navigation problem, where the historical example is used to compute a
semi-normal form.
In the final section, we study time minimal syntheses for $3d$-single 
input system, with terminal manifold of codimension $1$, in relation with 
the McKeithan network. 
Calculations are intricate and are handles using semi-normal forms.
We concentrate again on the abnormal case, in relation with continuity 
properties of the value function.

\section{General concepts and results from optimal control for single-input 
control system}
\label{sec:2}

We consider a smooth single-input control system of the form
\[
  \frac{\dd q(t)}{\dd t}=X(q(t))+u(t)\,Y(q(t)), \ q\in \mathbb{R}^n,
\]
where the control domain $U$ is either $\mathbb{R}$ or the interval $[-1,1]$
and the set of admissible control is the set of measurable mapping on 
$J=[0,t_f(u)]$ valued in $U$ and we denote by $q(\cdot, q_0,u)$
(in short $q(\cdot)$) defined on a subinterval of $J$ with $q(0)=q_0$.

\subsection{Maximum Principle}

Consider the problem of optimizing the transfer form $q_0$ to a smooth
submanifold $N$ of $\mathbb{R}^n$. Then the Maximum Principle tells us that
if a pair  $(u(\cdot),q(\cdot))$ is optimal on $[0,t_f]$, then there exists
an absolutely continuous vector function $p(\cdot)\in \mathbb{R}^n\setminus \{0\}$
such that if $H(q,p,u)=p\cdot (X(q)+u Y(q))$ denotes the Hamiltonian lift,
the following conditions are satisfied
\begin{enumerate}
  \item The triplet $(q,p,u)$ is solution a.e. on $[0,t_f]$ of 
    \begin{equation}
      \begin{aligned}
        &\dot q = \frac{\partial H}{\partial p},\quad 
        \dot p = -\frac{\partial H}{\partial q},
        \\
        &H(q,p,u) = \max_{v\in U} H(q,p,v).
      \end{aligned}
      \label{eq:pmp1}
    \end{equation}
  \item 
    $\bm{M}(q,p) = \max_{v\in U} H(q,p,v)$ is constant and equal to $-p^0$,
    where $p^0$ is non positive.
  \item The vector function $p(\cdot)$ satisfies at the final time the 
    transversality condition:
    \begin{equation}
      p(t_f) \perp T_{q(t_f)} N
      \label{eq:transversality}
    \end{equation}
\end{enumerate}

\begin{definition}
  A triplet $(q,p,u)$ solution of \eqref{eq:pmp1} is called an extremal and the 
  $q$-projection is called a geodesic. 
  If moreover it satisfies the transversality condition \eqref{eq:transversality}
  it is called a BC-extremal.
  An extremal is called singular if $H_Y(q(t),p(t))=p(t)\cdot Y(q(t))=0$
  a.e. on $[0,t_f]$.
  Assume $U=[-1,+1]$, a singular control is called strictly feasible 
if $u(\cdot)\in ]-1,+1[$, saturating at time $t_s$ if $|u(t_s)|=1$.
  An extremal control is called regular if it is given by $u(t)=\text{sign}
  (H_Y(q(t),p(t)))$ a.e. It is called bang-bang if the number of
  switches is finite. 
  An extremal is called abnormal (or exceptional) if $p^0=0$, so that
  from the Maximum Principle it is candidate to minimize or maximize 
  the transfer time.
  \label{def:extremal}
\end{definition}

\subsection{Computation of the singular controls}
If $Z_1,Z_2$ denote two (smooth) vector fields, the Lie bracket is computed 
with the convention:
$[Z_1,Z_2](q)= \frac{\partial Z_1}{\partial q}(q)Z_2(q)$
$-\frac{\partial Z_2}{\partial q}(q)Z_1(q)$
and if $H_1,H_2$ are the Hamiltonian lifts of $Z_1,Z_2$, it is related to
the Poisson bracket $\{H_1,H_2\}=\dd H_1(\vec{H}_2)$ by
$\{H_1,H_2\}(q,p)=p\cdot [Z_1,Z_2](q)$.
In the singular case, deriving twice with respect to time the equation
$H_Y(z(t))=0$, where $z(t)=(q(t),p(t))$, one gets:
\begin{equation}
  \begin{aligned}
    &H_Y(z(t))=\{H_Y,H_X\}(z(t))=0,
    \\
    &\{\{H_Y,H_X\},H_X\}(z(t)) + u(t)\, \{\{H_Y,H_X\},H_Y\}(z(t))=0.
    \label{eq:u-singular}
  \end{aligned}
\end{equation}
Hence, if $\{\{H_Y,H_X\},H_Y\}(z(t))$ is not identically zero the singular
extremals are given by the constrained Hamiltonian equation:
\begin{equation}
  \dot z(t) = H_X(z(t)) + u_s(t)\, H_Y(z(t)), \text{ with } 
  u_s(t) = -\frac{\{\{H_Y,H_X\},H_X\}(z(t))}{\{\{H_Y,H_X\},H_Y\}(z(t))}
\end{equation}
restricted to $H_Y=H_{[Y,X]}=0$.

\subsection{Action of the feedback pseudo-group $G_F$ \cite{dieudonne1971}}

Take a pair $(X,Y)$. The set of triplets $\{(\varphi, \alpha, \beta)\}$, 
where $\varphi$ is a local diffeomorphism and $u=\alpha(x)+ \beta(x) v$ with
$\beta \neq 0$ is a feedback, acts on the set of pairs $(X,Y)$ and this 
action defines the pseudo-feedback group $G_F$.
Each local diffeomorphism $\varphi$ can be lifted into a 
symplectomorphism using a Matthieu transformation and define an action of the
feedback group on \eqref{eq:u-singular} using the symplectomorphism 
only, one has see \cite{bonnard2003}.
\begin{theorem}
  The mapping $\lambda$ which yields for any pair $(X,Y)$ 
  the constrained differential equation \eqref{eq:u-singular}
  is covariant i.e. the following diagram is 
  commutative:
  \[ 
    \begin{tikzcd}
      (X,Y) \arrow{r}{\lambda} \arrow[swap]{d}{G_F} & \eqref{eq:u-singular} 
      \arrow{d}{G_F} \\
      (X',Y') \arrow[ur, phantom, "\scalebox{1.8}{$\circlearrowleft$}" description]\arrow{r}{\lambda}& (\ref{eq:u-singular}')
    \end{tikzcd}
  \]
  where (\ref{eq:u-singular}') consists of the equation \eqref{eq:u-singular}
  with the substitution $(X,Y)\leftarrow (X',Y')$.
  \label{thm:diagram}
\end{theorem}

\subsection{The $3d$-case}

Assume $n=3$. Introduce the following determinants :
\begin{align*}
  &D=\det(Y,[Y,X],\left[ \left[ Y,X \right],Y \right]),
  \\
  &D'=\det(Y,[Y,X],\left[ \left[ Y,X \right],X \right]),
  \\
  &D''=\det(Y,[Y,X],X).
\end{align*}
Using $H_Y=H_{[Y,X]}=H_{[[Y,X],X]}+u_s H_{[[Y,X],Y]}=0$, 
the singular control can be computed eliminating $p$ and 
depends on $q$ only and this leads to the following:
\begin{proposition}
  \begin{enumerate}
    \item 
      Assume $D$ nonzero, the singular controls are defined by the feedback
      $u_s(q)=-D'(q)/D(q)$ so that the corresponding geodesics are
      solutions of the vector field:
      $X_s(q)=X(q)+u_s(q)Y(q)$.
      Abnormal (exceptional) singular geodesics are contained in the 
      determinantal set $D''(q)=0$.
    \item The map $\lambda : (X,Y) \mapsto X_s$ is a covariant, restricting the
      action of the feedback pseudo-group to change of coordinates only.
  \end{enumerate}
  \label{prop:singular-control}
\end{proposition}

\subsection{High-order Maximum Principle in the singular case}

From \cite{krener1977}, in the singular case the generalized 
Legendre-Clebsch condition
\[
  \frac{\partial}{\partial u} \frac{\dd }{\dd t}\frac{\partial H}{\partial u}(q(t))
  =
  \left\{\left\{ H_Y,H_X \right\},H_Y  \right\}(q(t)) \ge 0
\]
is a necessary optimality condition.
This leads to the following.

\begin{proposition}
  In the $3$d-case, candidates to time minimizing are contained in $DD''\ge 0$ 
  and candidates to time maximizing are contained in the set
  $DD''\le 0$.
  If the corresponding inequalities are strict, they are respectively called
  hyperbolic or elliptic.
  \label{prop:high-order-pmp}
\end{proposition}

\section{Abnormal geodesics in planar Zermelo navigation problems}
\subsection{Notations and concepts}

Let $(M,g,F_{0})$ be a smooth navigation problem, where $(M,g)$ is a
Riemannian metric and $F_{0}$ is a vector field defining
the current. 
Our study is local and taking an orthonormal frame $(F_1,F_2)$ 
for the metric $g$, the problem can be written as a time minimal problem
for the system 
\begin{equation}
  \frac{\dd q(t)}{\dd t} =
  F_{0}(q(t))+\sum_{i=1,2} u_{i}(t)\,F_{i}(q(t)), \quad q=(x,y),
  \label{sys:F0F1F2}
\end{equation}
where the control $u=(u_1,u_2)$ is such that $\Vert u\Vert\leq 1$. 
The domain $M$ can be split into domain of strong current with
$\|F_0\|_g>1$ and weak current with $\|F_0\|_g<1$ and the 
case transition is called moderate with $\|F_0\|_g=1$.
Let $q_0$ be a point in the strong current domain.
Then the tangent model at $q_0$ is the cone of admissible directions
$F_0(q_0) + \sum_{i=1}^2 u_i F_i(q_0)$, $|u_i|\le 1$ whose boundary 
is limited by two directions called limit curves by Carathéodory.
They are precisely the two abnormal directions given by 
$H(q,p,u) = H_0(q,p) + \sum_{i=1}^2 u_i H_i(q,p)=0$,
where $p$ is the adjoint vector and $H_i(q,p)=p\cdot F_i(q)$, $i=0,1,2$.

The pseudo-group of local diffeomorphisms on the plane acts on the pair
$(F_0,g)$.
The metric can be set locally either in the isothermal form:
$g=a(x,y)\, (\dd x^2 +\dd y^2)$ or the polar form
$\dd r^2 + m(r, \theta)^2\, \dd \theta^2$, where the corresponding
coordinates are respectively called isothermal or polar coordinates.
In the case of revolution the isothermal and polar form becomes 
respectively $a(x)\, (\dd x^2+\dd y^2)$ and 
$\dd r^2+ m(r)^2\dd \theta^2$.

In the historical example, the metric is the Euclidean metric
$g=\dd x^2+ \dd y^2$, while the current is given by 
$F_0(q)=y \frac{\partial}{\partial x}$ and taking the line to the shore as
parallel, it is oriented along the parallel.

\subsection{Maximum Principle}

As explained in the introduction the geodesics curves can be parameterized
in two different ways using the Maximum Principle.

\subsubsection{Direct parameterization}
Maximizing the pseudo-Hamiltonian $H(q,u)$ using the constraints 
$|u|\le 1$ leads to the following:
\begin{proposition}
  Denoting $z=(q,p)$ one has:
  \begin{enumerate}
    \item The extremal controls are given by 
      $u_i= \frac{H_{i}(z)}{\sqrt{H_1(z)^2+H_2(z)^2}}$, $i=1,2$ so that
      $u_1^2+u_2^2=1$. 
    \item 
      The maximized Hamiltonian is given by 
      $\bm{M}(z)=H_{0}+\sqrt{H_1(z)^2+H_2(z)^2}$.
    \item 
      The maximized Hamiltonian $\bm{M}$ is constant and can be normalized to 
      $\{\pm 1, 0\}$ and the corresponding geodesics are hyperbolic
      if $\bm{M}=1$, elliptic if $\bm{M}=-1$ and the abnormal case corresponds 
      to $\bm{M}=0$.
  \end{enumerate}
  \label{prop:pmp-zermelo}
\end{proposition}

\subsubsection{Parameterization using the Carathéodory-Zermelo-Goh transformation}

Using the heading angle $\alpha$ of the ship amounts to set $u_1=\cos \alpha$ and
$u_2=\sin \alpha$ so that the pseudo-Hamiltonian takes the form
\[
  H(z) = H_0(z) +\cos \alpha H_1(z) + \sin \alpha H_2(z).
\]

The maximization condition leads to $\frac{\partial H}{\partial u}=0$.
Denoting $X(\tilde q)= F_0(q) + \cos \alpha F_1(q) + \sin \alpha F_2(q)$,
$Y(\tilde q)=\frac{\partial}{\partial \alpha}$,
where $\tilde q =(q, \alpha)$ denotes the extended state space. 
Using section \ref{sec:2}, one has.
\begin{proposition}
  Geodesics curves are solutions of the dynamics:
  \begin{equation}
    \dot{\tilde q} = X_s(\tilde q) = X(\tilde q) + u_s Y(\tilde q)
    \label{eq:dottildeq}
  \end{equation}
  with $u_s(\tilde q) = -D'(\tilde q)/D(\tilde q)$.
  \label{prop:singular-flot}
\end{proposition}

Next we introduce the following crucial set from the control point of view.

\begin{definition}
  Take $(Z_1,Z_2)$ two smooth vector fields in $\mathbb{R}^n$. 
  The collinear set is the feedback invariant set 
  $\mathcal{C}=\{q; \ Z_1(q) \text{ and } Z_2(q) \text{ are collinear.}\}$. 
  \label{def:collinear-set}
\end{definition}
One has clearly.
\begin{proposition}
  In the Goh extension, 
  \begin{enumerate}
    \item 
      the collinear set is defined by:
      $\{\ \exists \alpha; \ F_0(q) = \cos \alpha F_1(q) + \sin \alpha F_2(q)\},$
    \item 
      the geodesics curves solutions of \eqref{eq:pmp} are immersed curves
      outside the collinear set,
    \item 
      only abnormal geodesics can be non immersed curves when meeting the 
      collinear set $\|F_0\|_g=1$.
  \end{enumerate}
  \label{prop:collinear-set}
\end{proposition}

Before going further in our analysis let us analyze the case of revolution
with parallel current as a generalization of the historical example.

\subsubsection{The case of revolution with parallel current}

In polar coordinates, one has:
\[
  F_0(q) = \mu(r)\frac{\partial}{\partial \theta},\quad 
  g = \dd r^2 + m(r)^2\dd \theta^2
\]
so that 
\[
  F_1 = \frac{\partial}{\partial r},\quad 
  F_2 = \frac{1}{m(r)}\frac{\partial}{\partial \theta}.
\]
We define:
\begin{align*}
  &X = \cos \alpha \frac{\partial}{\partial r} 
  + \left(\mu(r)+ \frac{\sin \alpha}{m(r)}\right)
  \frac{\partial}{\partial \theta },
  \\
  &[Y,X] = \sin \alpha \frac{\partial}{\partial r} - \frac{\cos \alpha}{m(r)}
  \frac{\partial}{\partial \theta },
  \\
  &\left[ \left[ Y,X \right],X \right] 
  = \left(-\mu'(r)\sin \alpha + \frac{m'(r)}{m(r)^2}\right) 
  \frac{\partial}{\partial \theta },
  \\
  &\left[ \left[ Y,X \right],Y \right] 
  = \cos \alpha \frac{\partial}{\partial r} + \frac{\sin \alpha}{m(r)} 
  \frac{\partial}{\partial \theta }.
\end{align*}

\begin{lemma}
  Computing we have:
  \begin{itemize}
    \item 
      $D=\frac{1}{m(r)}$,
    \item 
      $D'=-\mu'(r)\sin^2 \alpha + \frac{m'(r)}{m(r)^2}\sin \alpha$,
    \item
      $D''=\mu(r)\sin \alpha + \frac{1}{m(r)}$.
  \end{itemize}
  Hence $D$ is non zero.
  \label{lem:Ds}
\end{lemma}
This yields the following proposition.
\begin{proposition}
  In the case of revolution, with parallel current one has:
  \begin{enumerate}
    \item 
      The pseudo-Hamiltonian in the $\tilde q$-representation takes the form:
      \[
        H=p_r\cos \alpha + p_\theta \left(\mu(r)+\frac{\sin \alpha}{m(t)}+p^0\right).
      \]
    \item The Clairaut relation is satisfied i.e. $p_\theta$ is constant and 
      moreover 
      \begin{equation}
        p_\theta \left( \mu(r) + \frac{1}{m(r)\sin \alpha} \right)+ p^0 =0.
        \label{eq:clairaut}
      \end{equation}
    \item 
      The geodesics equations 
      $\dot{\tilde q} = X(\tilde q) -\frac{D'(\tilde q)}{D(\tilde q)}Y(\tilde q)$
      can be integrated by quadratures, solving the implicit equation
      \eqref{eq:clairaut} to integrate the dynamics of the heading 
      angle $\alpha$.
    \item 
      Singular points for the geodesics dynamics occur only restricting 
      to abnormal geodesics in $D''=0$ when $D'=0$.
  \end{enumerate}
  \label{prop:quadratures}
\end{proposition}

In the historical example a cusp singularity was observed in \cite{caratheodory1965}
and it will serve  as a model to analyze the general case in the frame of 
singularity theory, since integrability is not a technical 
requirement. One needs to recall some elementary facts.

\subsection{A brief recap about cusp singularity theory for geodesics \cite{thom1989}}

In our problem, one considers a geodesic curve $t\mapsto \sigma(t)$
defined on $J$ and meeting $\|F_0\|_g=1$ at $t=t_0$.
Making a time translation, one can take $J=[-t_0,t_0]$
so that $\sigma$ touches the boundary at $t=0$, so that
$\dot \sigma(0)=0$.

\begin{definition}
  The point $\sigma(0)$ is a cusp point of order $(p,q)$,
  $2\le p\le q$ if $\sigma^{(p)}(0)$ and 
  $\sigma^{(q)}(0)$ are independent.
  The point $\sigma(0)$ is called an ordinary cusp 
  (or a semicubical point) if $p=2$, $q=3$, and a 
  ramphoid cusp if $p=2$, $q=4$.
  \label{def:cusp}
\end{definition}

\subsubsection{Semicubical point}
From \cite[p.~56]{walker1978}, an algebraic model in $\mathbb{R}[x,y]$
at $\sigma(0)=0$ is given by the equation $x^3-y^2=0$.
Moreover it is the transition between a $\mathbb{R}$-node solution
of the equation $x^{3}-x^{2}+y^{2}=0$, where the origin is a double point
with two distinct tangents at $0$: $x\pm y=0$ and a $\mathbb{C}-$node 
solution of $x^{3}+x^{2}+y^{2}=0$ with two complex tangents at $0$ given by
$x\pm iy=0$ and with two distinct components $x=y=0$ and a smooth 
real branch.

A neat description from singularity theory suitable in our analysis is given
by \cite[p.~65]{arnold1991} and is associated to a typical perestroika
of a plane curve depending on a parameter and having a semicubical
cusp point for some value of the parameters:
\begin{figure}[ht]
  \centering
  \begin{tabular}{ccc}
    \includegraphics[scale=0.27]{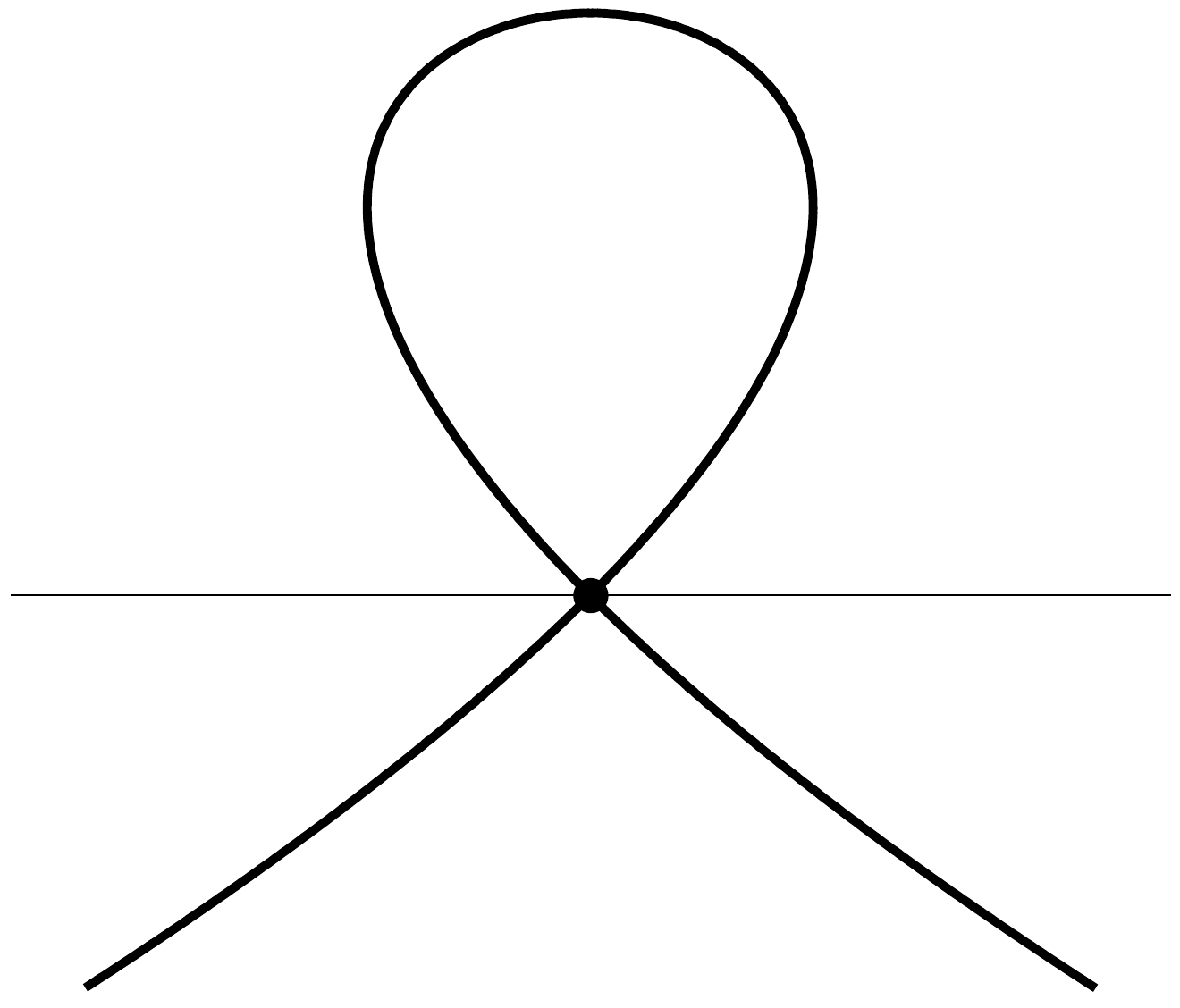}
    &
    \includegraphics[scale=0.27]{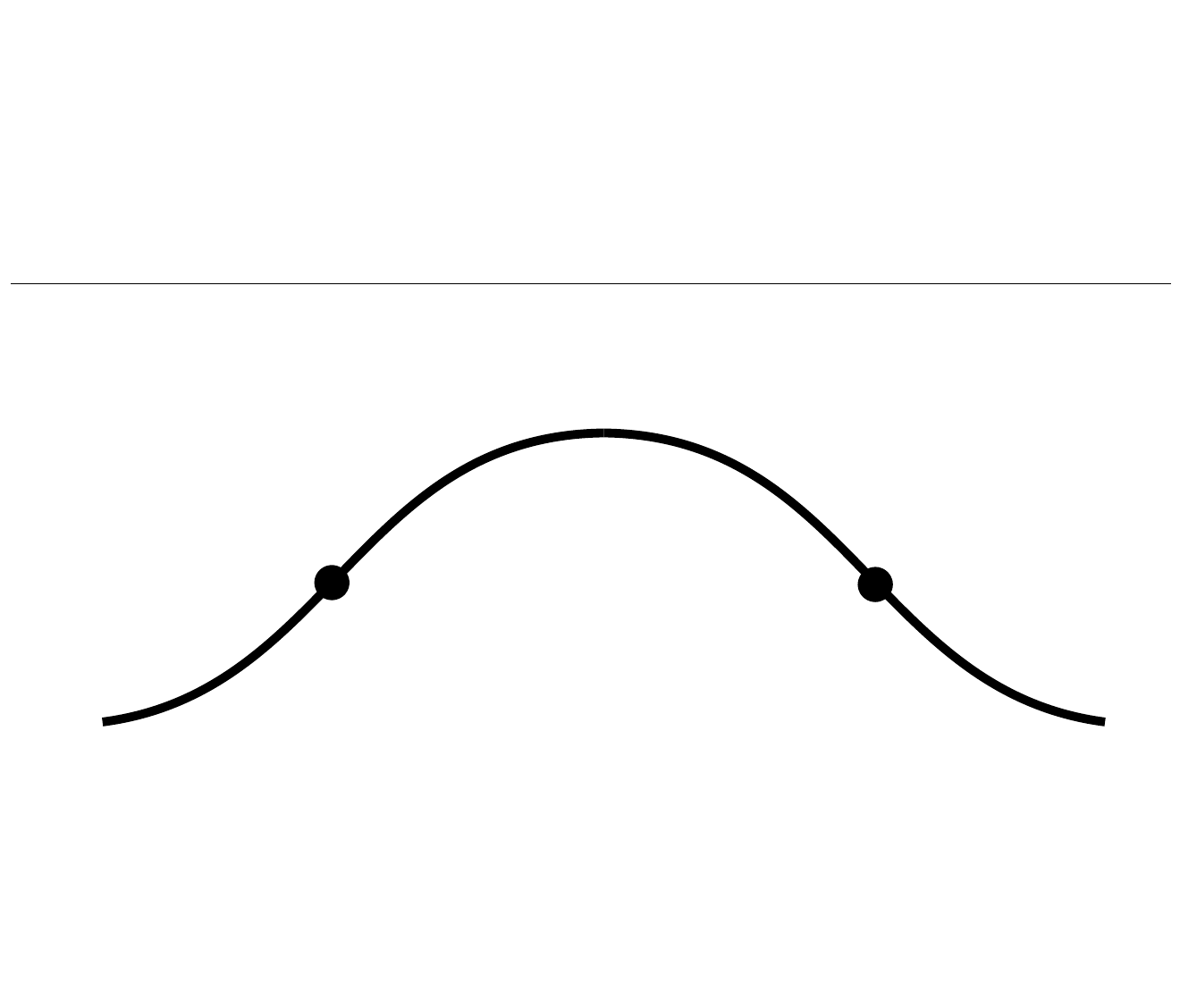}
    &
    \includegraphics[scale=0.27]{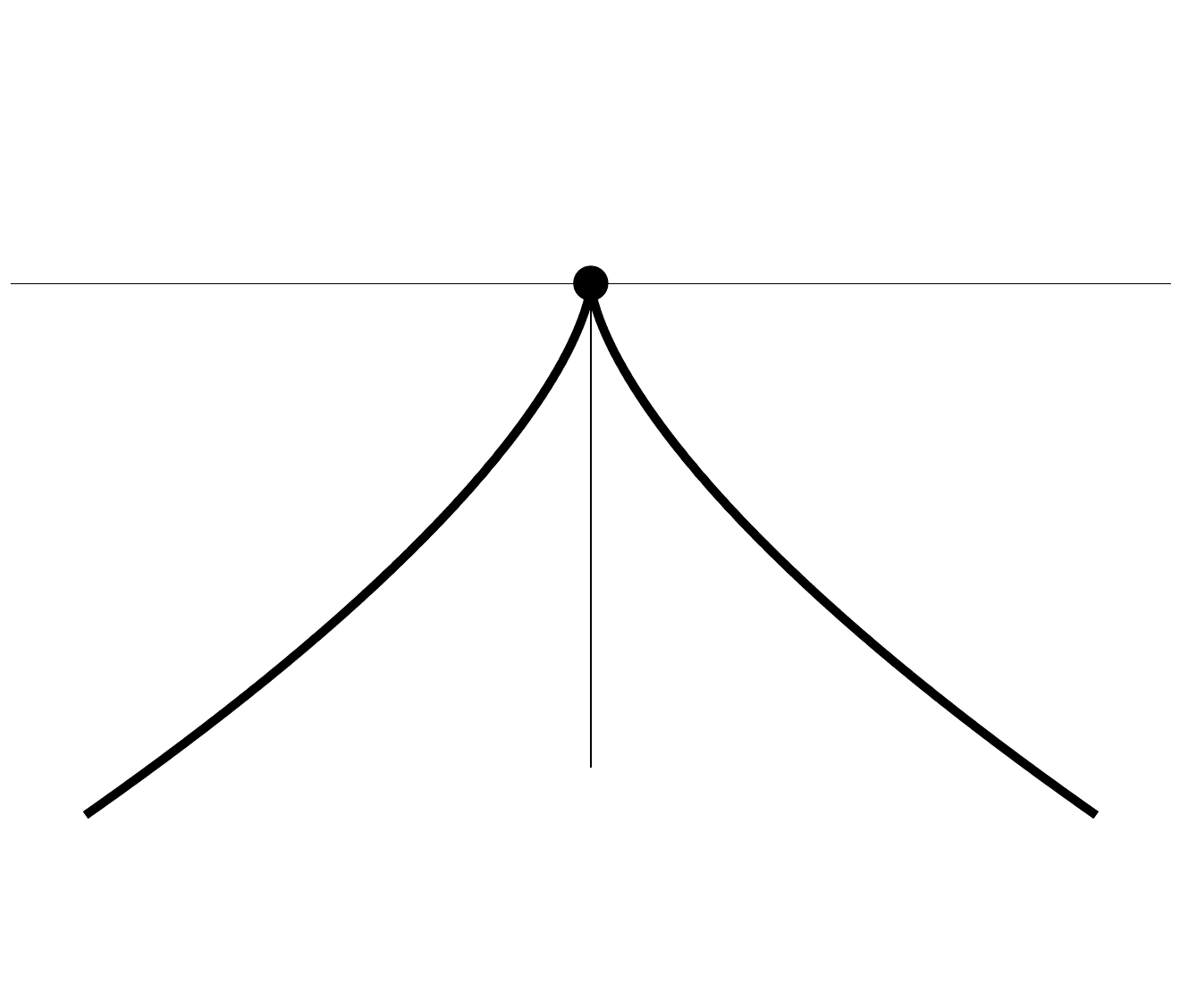}
  \end{tabular}
  \caption{Unfolding semicubical cusp.}
  \label{fig:nodes}
\end{figure}
where the curves sweep an umbrella while their inflectional tangents sweep
another umbrella surface.

\subsubsection{Semicubical unfolding in the historical example}

In the historical example the geodesics equation is given by 
\begin{equation}
  \frac{\dd x}{\dd t}=y+\cos\alpha,\quad
  \frac{\dd y}{\dd t}=\sin\alpha,\quad
  \frac{\dd\alpha}{\dd t}=-1 + \alpha^2.
\end{equation}

The boundary of moderate current is taken as $y=-1$ and making the 
translation $Y=y+1$and expanding at $\alpha=0$ up to order $2$,
the system takes the form

\begin{equation}
  \frac{\dd x}{\dd t}=Y-\alpha^{2}/2,\quad \frac{\dd Y}{\dd t}=\alpha, \quad 
  \frac{\dd\alpha}{\dd t}=-1+\alpha^{2}.
  \label{eq:zermelo-model-app}
\end{equation}
and take a point $q_0=(x_0, y_0, z_0)$ in a neighbourhood of $0$ in the strong
current domain $Y<0$ and let $t\mapsto \sigma(t)$ be a geodesic curve with
$\sigma(0)=(x_0,y_0,z_0)$.

Then one has:
\begin{proposition}
  Fixing $q_0$ and considering the geodesics passing through $q_0$, we have:
  \begin{enumerate}
    \item 
      The abnormal geodesic meets the boundary at a semicubical cusp 
      with vertical tangent.
    \item 
      Hyperbolic geodesics are self-intersecting curves corresponding 
      to a $\mathbb{R}$-node.
    \item 
      Elliptic geodesics exist only in the strong current domain $Y<0$
      and correspond to a $\mathbb{C}$-node.
  \end{enumerate}
  \label{prop:rnode}
\end{proposition}

Hence geodesics curves form an unfolding of the semicubical cusp with 
one parameter depending upon the initial heading angle $\alpha_0$,
see Fig.\ref{fig:versal-unfolding}.

\begin{figure}[ht]
  \centering
  \includegraphics[scale=0.6]{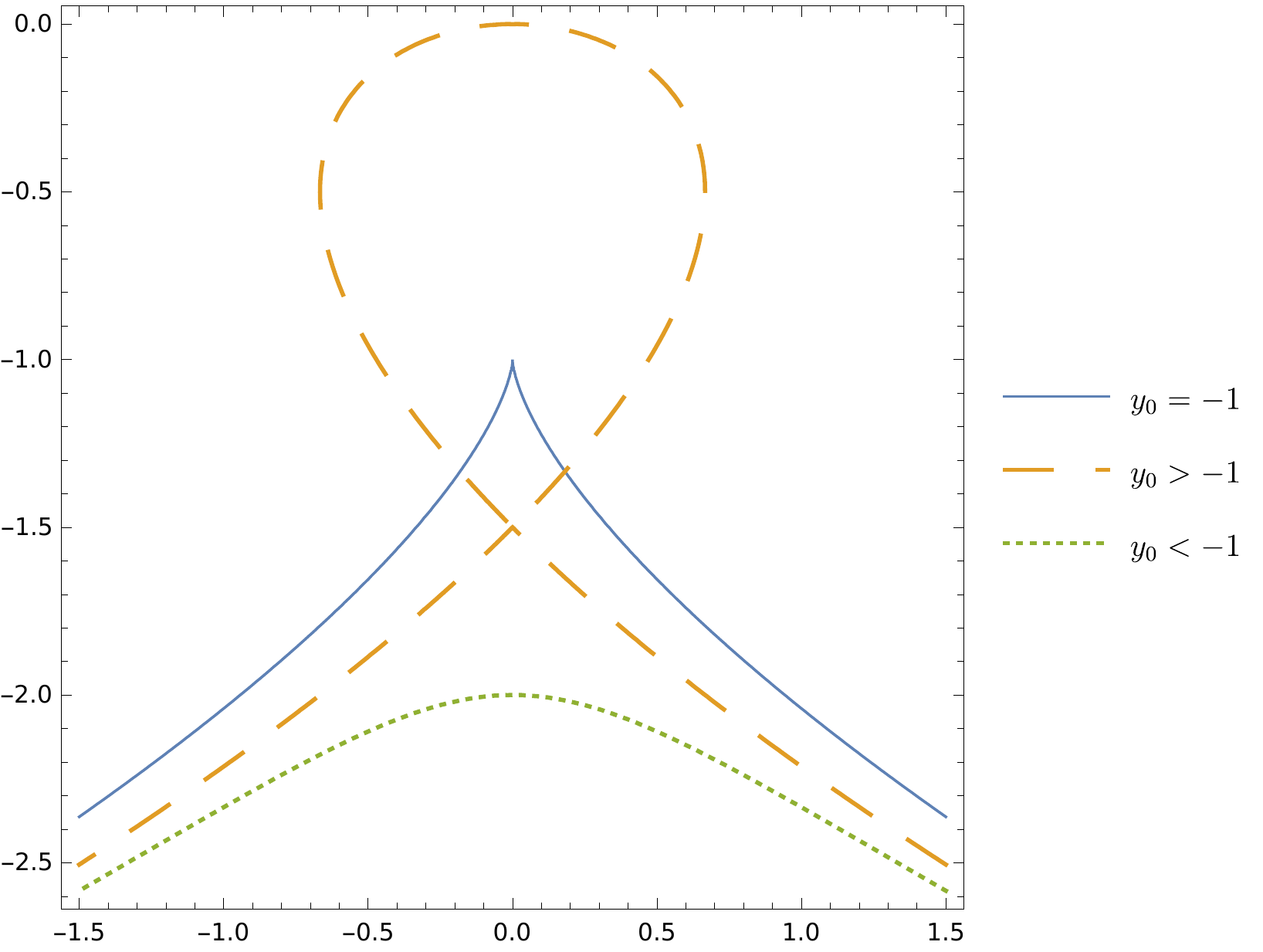}
  \caption{Quickest nautical path as a miniversal unfolding of the generic singularity
  of the abnormal geodesic.}
  \label{fig:versal-unfolding}
\end{figure}

\subsection{The analysis of the geodesics curves near the set $\|F_0\|_g=1$
and regularity of the time minimal value function.}

\begin{proposition}
  Let $(M,g,F_0)$ be a two dimensional Zermelo navigation problem and 
  $\tilde q_1=(q_1, \alpha_1)$ be a point 
  in the collinearity set $\mathcal{C}$.
  Assume that:
  \begin{itemize}
    \item 
      the $q$-projection of $\mathcal{C}$ is a regular curve at $q_1$ 
    \item the geodesic  $\sigma(\cdot)$ is not an immersion at $q_1$.
  \end{itemize}
  Consider $\tilde \sigma : t \mapsto \tilde \sigma(t)\coloneqq (\sigma(t), \alpha(t))$, 
  $t\in [t_1,0],\, t_1<0$ to be the geodesic 
  passing through $\tilde q_1$ at $t=0$ satisfying
  \begin{equation}
    \dot{\tilde \sigma} = X(\tilde \sigma) -\frac{D'(\tilde \sigma)}{D(\tilde \sigma)}Y(\tilde \sigma),
    \label{eq:sigmatilde}
  \end{equation}
  where 
  $X = F_0+ \cos \alpha F_1  + \sin \alpha F_2$, $Y=\frac{\partial}{\partial \alpha}$, 
  $D = \det(Y,[Y,X],[[Y,X],Y])$ and $D' = \det(Y,[Y,X],[[Y,X],X])$.\\
  Then we have the two cases:
  \begin{enumerate}
    \item 
      $\dot \alpha(0) \neq 0$: $\sigma$ has a semicubical cusp at $q_1$.
    \item 
      $\dot \alpha(0) =0$: $\tilde q_1$ is a singular point with a spectrum 
      $\{0, \pm \lambda\}$ or $\{0, \pm i\, \gamma\}$.
  \end{enumerate}
  \label{prop:cusp}
\end{proposition}

\begin{proof}
  {\it Normalization.}
  The problem is local in a neighbourhood of $\tilde q_1$ and 
  it is enough to show the proposition for $q_1=0$.
  We can choose a coordinate system $(x,y)$ to normalize, at the point $q_1$, 
  the vector field $F_0$ along the direction $-\frac{\partial}{\partial x}$ i.e.
  we take two smooth functions $b$ et $c$ such that
  \[
    F_0(x,y) = b(x,y)\,\frac{\partial}{\partial x}+c(x,y)\,\frac{\partial}{\partial y},
  \]
  with $b(0,0)=-1$ and $c(0,0)=0$. 
  Then a frame $(F_1,F_2)$, orthonormal with respect to the metric $g$ 
  set in the isothermal form  $g(x,y)=a(x,y)\, \left( \dd x^2+\dd y^2 \right)$,
  can be taken in such way that $F_0$ and $F_1$ has opposite direction at $q_1$: 
  \[
    (F_1,F_2)=\left(\frac{1}{\sqrt{a(x,y)}}\, \frac{\partial}{\partial x},
    \frac{1}{\sqrt{a(x,y)}}\, \frac{\partial}{\partial y}\right),
  \]
  where $a$ is a smooth positive function. \\
  In a neighbourhood of $q_1$ we write
  $a(x,y)=\sum_{1\le i,j\le k} a_{ij}\, x^i y^j + \varepsilon_1(x,y)$,
  $b(x,y)=\sum_{1\le i,j\le k} b_{ij}\, x^i y^j + \varepsilon_2(x,y)$,
  $c(x,y)=\sum_{1\le i,j\le k} c_{ij}\, x^i y^j + \varepsilon_3(x,y)$,
  where $\varepsilon_1, \varepsilon_2, \varepsilon_3$ are terms of order higher
  than $k$.
  
  The projection of the collinearity set  is 
  \begin{equation*}
    \begin{aligned}
      \{\|F_0\|_g=1\} = \{
        a_{00}-1 + &(a_{10} - 2 b_{10}) x + (a_{01} - 2 b_{01}) y +
        \left(c_{01}^2-\kappa_2 \right)\, y^2\\
        &+
        2(c_{01} c_{10}-\kappa_1)\, x y
        + \left(c_{10}^2-\kappa_2'\right)\, x^2 
        + o(|x,y|^2) = 0
      \},
      \label{eq:colset}
    \end{aligned}
  \end{equation*}
  and is regular near $q_1$ and its tangent at $q_1$ can be normalized 
  to the horizontal line $y=0$ with $a_{00}=1$,  $a_{10} = 2 b_{10}$ and
  the constants $\kappa_1=-a_{11}/2+3 b_{01} b_{10}+ b_{11}$ 
  $\kappa_2=-a_{02}+3 b_{01}^2+2 b_{02}$,
  and
  $\kappa_2'=-a_{20}+3 b_{10}^2+2 b_{20}$
  will have some importance in the
  sequel.
  Due to the normalization of $F_0$ and $F_1$, $\alpha_1$ has to be equal to $0$.

  {\it Computation.}
  The geodesic $\sigma(\cdot)$ is not an immersion at $q_1$ since
  $\dot \sigma(0)=F_0(q_1)+\cos \alpha_1 F_1(q_1) + \sin \alpha_1 F_2(q_1)=0$ and
  denoting by $p(\cdot)$
  the corresponding adjoint vector, we have
  $\bm{M}(q_1,p(0))=p(0)\cdot \dot \sigma(0)=0$ and $\sigma(\cdot)$
  is an abnormal geodesic.

  Integrating \eqref{eq:sigmatilde}, $\tilde \sigma(\cdot)$ can be parameterized as
  \begin{equation}
    \left\{
      \begin{array}{l}
        \sigma(t)=\left(-\frac{1}{3} t^3 \delta^2+o(t^3),\frac{1}{2} t^2 \delta+o(t^2)\right) 
        \\
        \alpha(t) = t \delta + o(t)
      \end{array},
    \right.
    \label{eq:jetsigmatilde}
  \end{equation}
  where $\delta \coloneqq a_{01}/2 - b_{01}$. 
  
  The expansions at $\tilde q_1=0$ of the  determinants 
  $D(\tilde q)$ and $D'(\tilde q)$ are 
  \begin{equation}
    \begin{aligned}
      D(\tilde q) &= 1 -2 b_{01} y-2 b_{10} x +o(x,y, \alpha),
      \\
      D'(\tilde q) &=   
      y \left(4 a_{01} \delta-3 \delta^2+\kappa_2\right)+x (5
      b_{10} \delta+\kappa_1)+c_{01} z-\delta
      + o(x,y, \alpha).
      \label{eq:DDp}
    \end{aligned}
  \end{equation}

%
%

  {\it Case $\dot \alpha(0)\neq 0$.}
    
  In this case $a_{01}/2- b_{01}\neq 0$ and from \eqref{eq:jetsigmatilde} $\sigma(\cdot)$ 
  has a semicubical cusp at $q_1$.

  {\it Case $\dot \alpha(0)=0$.}
  $\tilde q_1$ is a singular point of the system 
  $\dot{\tilde q} = X_s(\tilde q)$.
  From \eqref{eq:DDp}, for $\delta=0$, $\tilde \sigma \equiv 0$ is the integral curve 
  of $X_s$ (Lipschitz) with $\tilde \sigma(0)=0$, therefore $\tilde \sigma$ is reduced to $0$.  
  The characteristic polynomial $\chi$ of the Jacobian matrix 
  $\frac{\partial X_s}{\partial \tilde q}$ evaluated at $\tilde q_1$ is 
  \[
    \chi(s) = s\,\left(\lambda - s^2\right),\quad 
    \lambda=c_{01}^2 -\kappa_2,
  \]
  hence the spectrum is $\{0, \pm \sqrt{\lambda}\}$ if $\lambda\ge 0$ or $\{0, \pm i \sqrt{-\lambda}\}$ otherwise.
\end{proof}

\begin{figure}[htpb]
  \centering
  \includegraphics[width=0.55\columnwidth]{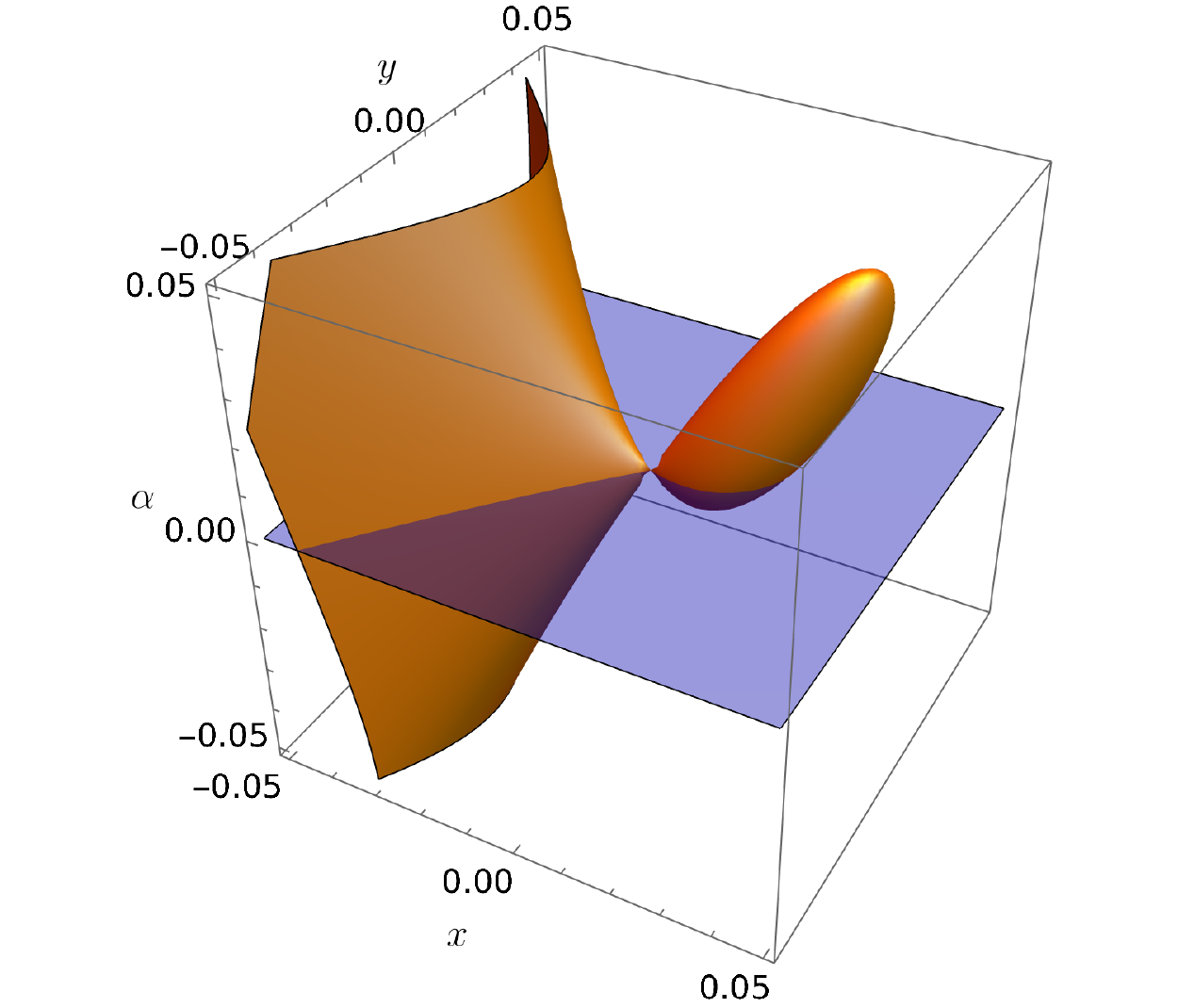}
  \includegraphics[width=0.40\columnwidth]{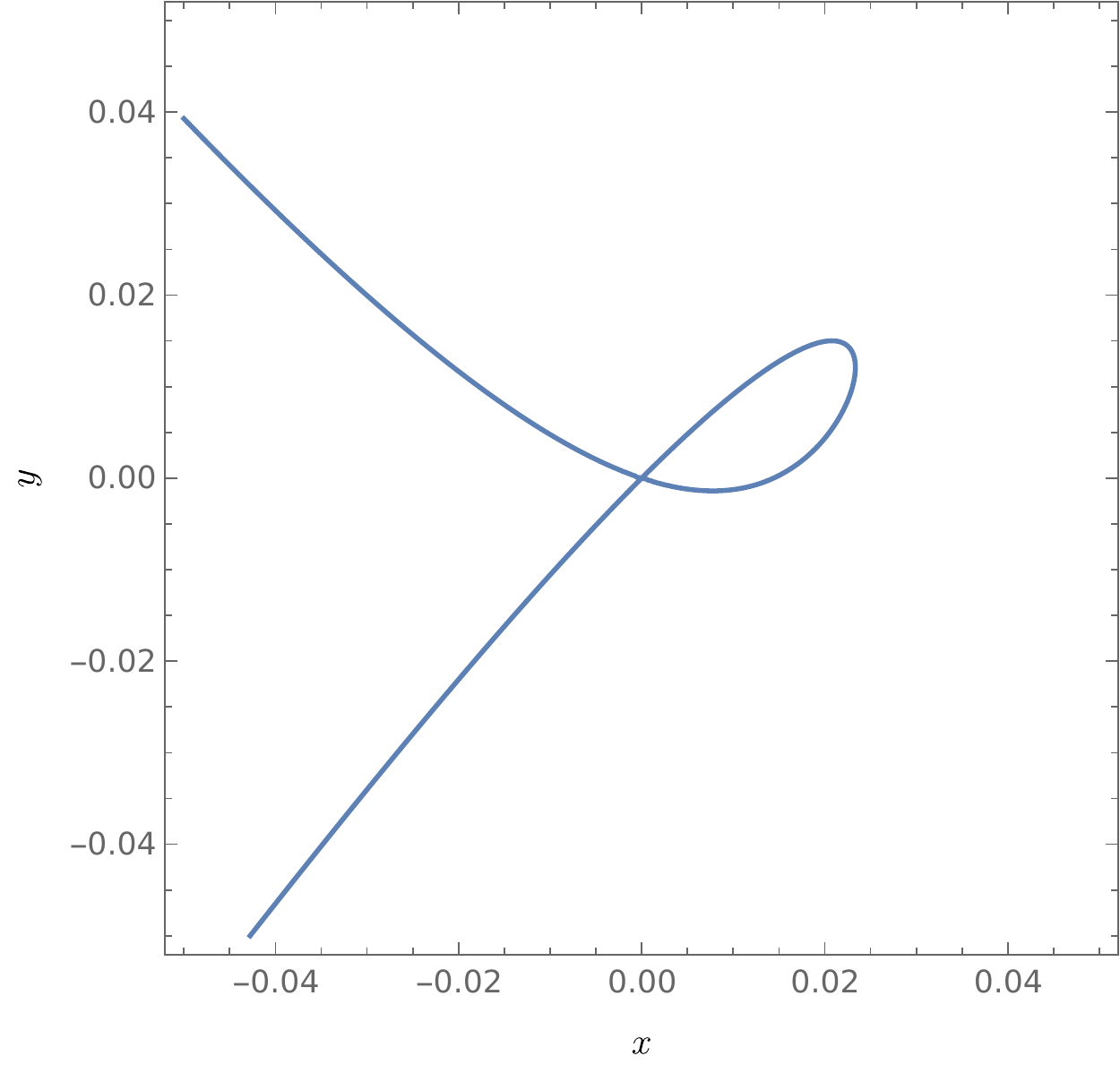}
  \caption{{\it (left)} Surface $D''(\tilde q)=0$ for the semi-normal form constructed in the 
  proof of Proposition \ref{prop:cusp} for $\delta=0$ and random values for the 
  coefficients $a_{ij},b_{ij},c_{ij}$. {\it (right)} Intersection of $D''=0$ with $\alpha=0$.}
  \label{fig:ddpp}
\end{figure}

\begin{remark}
\begin{itemize}
\item 
The spectrum has resonance and in the general case, this leads to moduli
in the classification.
But since the vector field is 
geodesic, one has a foliation related to the set $D''=0$, 
while $DD''>0$ and $DD''<0$ 
so that we expect a complete
classification of the geodesic flow.

\item
Using the semi-normal form constructed in Proposition \ref{prop:cusp},
  the expansion at $\tilde q_1=0$ of the  determinant
  $D''(\tilde q)$ up to order $2$ is 
  \begin{equation*}
    \begin{aligned}
      D''(\tilde q) &=  
      \frac{1}{2} \Big(y \left(\delta (4 a_{01} y+8 b_{10} x-2)-3
        \delta^2 y+2 \kappa_1 x+\kappa_2 y\right)
        \\
        &\hspace{4.5cm}+2 z (c_{01}
      y+c_{10} x)+\kappa_2' x^2+z^2\Big)
      +o(|x,y, \alpha|^2),
    \end{aligned}
  \end{equation*}
   and for $\dot \alpha=0$ the surface $D''(\tilde q)=0$
   is not regular at $\tilde q_1$ (see Fig. \ref{fig:ddpp}).
   \end{itemize}
\end{remark}

The following theorem describes the optimality properties of the abnormal
and hyperbolic geodesics.
\begin{theorem}
  Let $\tilde q_1=(q_1, \alpha_1)$ such that $q_1$ is a semicubical cusp at $t=0$ 
  for the abnormal geodesic $\sigma_a(\cdot)$.
  There exists an neighbourhood $V$ of $q_1$, a point $q_0$ in $V\cap \sigma_a(\cdot)$ in which  we have:
  \begin{enumerate}
    \item 
      The abnormal arc is optimal up from $q_0$ to the cusp point included.
    \item 
      Self-intersecting geodesics starting from $q_0$ in a conic
      neighbourhood of $\alpha_1$ are optimal up to the intersection
      point $q_1$ with the abnormal.
    \item The value function $T:q_f\mapsto T(q_0,q_f)$ is discontinuous
      for each $q_f\neq q_1$ on the abnormal geodesic $\sigma_a(\cdot)$.
  \end{enumerate}
  \label{thm:value-function}
\end{theorem}
\begin{proof}
  For $\xi$ in the extented state space, we use the notation
  $\sigma_h(\cdot, \xi)$ (resp. $\sigma_a(\cdot,\xi)$) for the projection
  on the $q$-space of the hyperbolic (resp. exceptional) extremal $\tilde \sigma(\cdot,\xi)$ passing through $\xi$ at time $0$.
  The point $\tilde q_1$ can be identified to $0$ and we use the same 
  normalization as in the proof of Proposition \ref{prop:cusp}, 
  which leads to consider the semi-normal form 
  \begin{align*}
    &F_0(x,y) = b(x,y)\,\frac{\partial}{\partial x}+c(x,y)\,\frac{\partial}{\partial y},
    \\
    &F_1= \frac{1}{\sqrt{a(x,y)}}\, \frac{\partial}{\partial x},
    \quad F_2=\frac{1}{\sqrt{a(x,y)}}\, \frac{\partial}{\partial y},
  \end{align*}
  with
  $a(x,y)=\sum_{1\le i,j\le k} a_{ij}\, x^i y^j + \varepsilon_1(x,y)$,
  $b(x,y)=\sum_{1\le i,j\le k} b_{ij}\, x^i y^j + \varepsilon_2(x,y)$,
  $c(x,y)=\sum_{1\le i,j\le k} c_{ij}\, x^i y^j + \varepsilon_3(x,y)$
  where $\varepsilon_1, \varepsilon_2, \varepsilon_3$ are terms of order higher
  than $k$,
  $a_{00}=1$, $b_{00}=-1$, $c_{00}=0$, $a_{10}=2\, b_{10}$
  and $a_{01}-b_{01}\neq 0$.

  The proof goes as follows. In relation with Fig.\ref{fig:value-fct},
  we define from $\tilde q_1$ the points $q_0$ and $q_2$ on the abnormal geodesics 
  as $\sigma_a(t_0, \tilde q_1)$ and $\sigma_a(t_2, \tilde q_1)$ 
  respectively for some given $t_0<t_2<0$.
  In the extended space we take $\tilde q_2'=(q_2, \alpha_2')$ and a time 
  $t_1<0$ such that $q_0$ is reached from $q_2$ by a hyperbolic geodesics in 
  time $t_1$.
  The following computation aims to express
  $\alpha_2'$ and $t_1$ in terms of $t_0, t_2$.

  More precisely, 
  for small nonpositive times $t_0,t_1,t_2$ and small angle $\alpha_2'$,
  we expand 
  $q_2\coloneqq \sigma_a(t_2,\tilde q_1)$ and 
  $\sigma_h(t_1,\tilde q_2')$ up to order $3$ and we obtain
  \begin{align*}
    \tilde \sigma_a(t_2&,\tilde q_1)\!=\! 
    \Big(-\frac{\delta^2}{3} t_2^3,\frac{\delta
      }{2} t_2^2,\frac{\delta}{12} t_2 \big(-5 a_{01}
        \delta t_2^2+2 c_{01}^2 t_2^2-6 c_{01} t_2+4
        c_{10} \delta t_2^2&&
        \\
        & \hspace{7cm}+2 \delta^2 t_2^2-2 \kappa_2
    t_2^2+12\big)\Big) +o(t_2^3),&&
    \\
    \tilde \sigma_h(t_1&,\tilde q_2')\!=\!\Big(
      \frac{1}{2} \left(2 \alpha_2' t_1+\delta t_1^2+\delta t_2^2\right),
      \frac{1}{6} \big(-3 \alpha_2'^2 t_{1}-\!6 \alpha_2'
        \delta t_{1}^2-2 \delta^2 t_{1}^3-\!3 \delta^2 t_{1}
      t_2^2-2 \delta^2 t_2^3\big),
      \\
      &\frac{1}{2} \left(2
        \alpha_2' t_{1}+\delta t_{1}^2+\delta
      t_2^2\right)+\frac{1}{2} c_{01} \delta t_{1}
      t_2^2,\frac{1}{12} \big(3 a_{01} \alpha_2'^2 t_{1}
        -15a_{01} \alpha_2' \delta t_{1}^2-5 a_{01} \delta^2
        t_{1}^3
        \\
        &-18 a_{01} \delta^2 t_{1} t_2^2+12
        \alpha_2'^2 c_{10} t_{1}-12 \alpha_2'^2 \delta t_{1}
        +6\alpha_2' c_{01}^2 t_{1}^2+12 \alpha_2' c_{10} \delta
        t_{1}^2+6 \alpha_2' \delta^2 t_{1}^2
        \\
        &-6 \alpha_2' \kappa_2
        t_{1}^2+2 c_{01}^2 \delta t_{1}^3+4 c_{10}
        \delta^2 t_{1}^3
        +2 \delta^3 t_{1}^3+18 \delta^3
        t_{1} t_2^2-2 \delta \kappa_2 t_{1}^3-6 \delta
      \kappa_2 t_{1} t_2^2\big)
      \\
      &-\alpha_2' c_{01}
      t_{1}+\alpha_2'-\frac{1}{2} c_{01} \delta
      t_{1}^2+\delta t_{1}
    \Big) + o(|t_1,t_2, \alpha_2'|^3).
  \end{align*}
  Computing, the equation \[q_0\coloneqq \sigma_a(t_0,\tilde q_1)
  =\sigma_h(t_1,\tilde q_2')\]
  is satisfied up to order $2$ in $t_0,t_1,t_2, \alpha_2'$, for
  \begin{equation*}
    \begin{aligned}
      &t_1 = t_2-t_0-2\sqrt{t_0^2+t_0t_2+t_2^2},
      \\ 
      &\alpha_2' = \frac{\delta}{2} \frac{t_0^2-t_2^2-t_1^2}{t_1},
    \end{aligned}
    \label{eq:q0-tp}
  \end{equation*}
  since $t_0<t_2$.

  Finally, we compare the cost of the abnormal geodesic from $q_0$ to $q_1$,
  which is $-t_0>0$ and the cost $-t_1-t_2$ of the concatenation of the
  hyperbolic geodesic from $q_0$ to $q_2$
  and the abnormal arc from $q_2$ to $q_1$ as
  follows:
  \begin{equation*}
    \begin{aligned}
      -t_1-t_2 &=
      2 \left( \sqrt{(t_0-t_2)^2+3t_2 t_0}+t_0-t_2\right)-t_0
      \\
      &\ge
      2 \left( |t_0-t_2|+t_0-t_2\right)-t_0
      \\
      &\ge -t_0.
    \end{aligned}
  \end{equation*}
  This shows that the value function $q_f\mapsto T(q_0,q_f)$ is  discontinuous if $q_f$
  is on the abnormal geodesic and $q_f\neq q_1$.
\end{proof}

\begin{figure}
  \centering
  \includegraphics[scale=0.5]{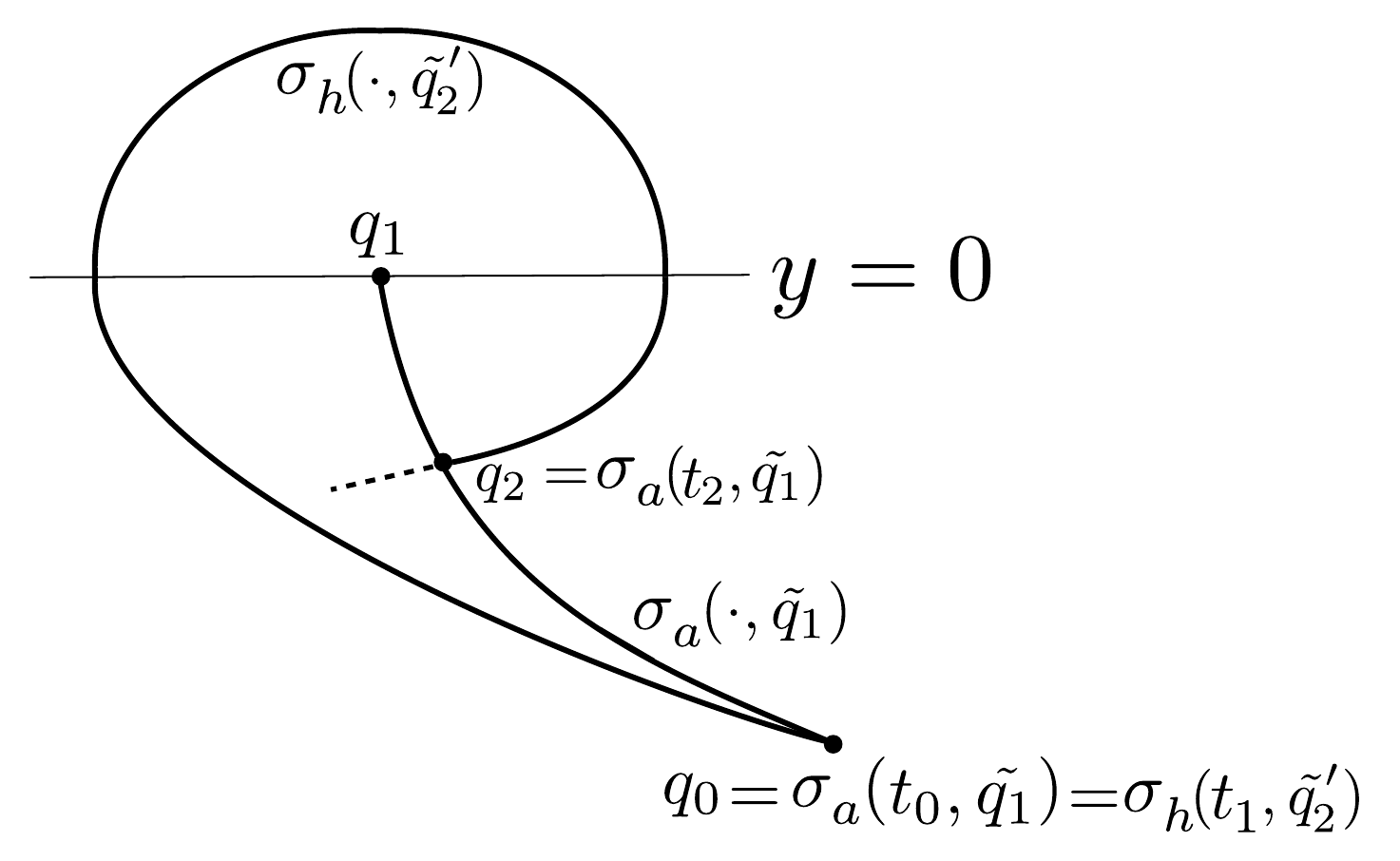}
  \caption{Hyperbolic and abnormal geodesics in a neighbourhood
  of the collinearity set.}
  \label{fig:value-fct}
\end{figure}

\section{Time minimal exceptional geodesics in optimization of chemical networks}
\subsection{A brief recap about the optimal control of chemical networks}

In this section, we introduce the concepts for the optimization of chemical 
networks, see \cite{bonnard2019}.
In particular we shall consider the McKeithan network:
$
\tikz[baseline={(a.base)},node distance=6mm] {
  \node (a) {$T+M$};
  \node[right=of a] (b) {$A$};
  \node[right=of b] (c) {$B$};
  \draw (a.base east) edge[left,->] node[above] {{$\scriptstyle k_1$}} (b.base west) (b.base east) edge[->] node[above] {{$\scriptstyle k_2$}} (c.base west);
  \draw (a) edge[bend right=40,<-] node[below right] {{$\scriptstyle k_3$}} (b);
  \draw (a.south) edge[bend right=60,<-] node[below right] {{$\scriptstyle k_4$}} (c);
}
$.
The state space is formed by the concentration vector:
\[
  c = (c_T,c_M,c_A,c_B)
\]
of the respective chemical species.
We note $\delta_1=c_T+c_A+c_B$, $\delta_2=c_M+c_A+c_B$ the first integrals 
associated to the 
dynamics and let $x=c_A$, $y=c_B$, then the system is described by the equation:
\begin{align*}
  &\dot x = -\beta_2 x v^{\alpha_2}-\beta_3 x v^{\alpha_3}-\delta_3 v\, (x+y)
  +\delta_4 v+v\, (x+y)^2
  \\
  &\dot y = \beta_2 x v^{\alpha_2}-\beta_4 y v^{\alpha_4}
  \\
  &\dot v = u, \qquad |u|\le 1,
\end{align*}
with
\begin{align*}
  0\le x\le \delta_1, \ 0\le y\le \delta_2, \
  \delta_3=\delta_1+\delta_2, \ \delta_4=\delta_1 \delta_2,
\end{align*} 
the Arrhenius law gives 
$k_i = A_i \exp(-E_i/RT),\ i=1,2,3,4$ 
($E_i$ is the activation energy, $T$ is the temperature, $A_i,R$ are
constant)
and 
\[
  v=k_1,\ k_2 = \beta_2 v^{\alpha_2}, \ k_3 = \beta_3 v^{\alpha_3},
  k_4 = \beta_4 v^{\alpha_4}.
\]
Maximizing the production of the $A$ species leads to a time minimal
control problem with a terminal manifold $N:x=d$, $d$ being the
desired production of $A$. 

We denote by $\dot q = X+u\, Y$, $|u|\le 1$ the control system.

The singular geodesics are solutions of the dynamics 
\[
  \dot q = X+ u_s\, Y,\quad u_s=-D'/D.
\]
Each optimal solution is a concatenation by arcs $\sigma_+, \sigma_-$, 
where the control is $u=\pm 1$ and singular arcs $\sigma_s$.
In this case, the complexity of the surface $D,D',D''$ 
 contrasts with those of the Zermelo navigation problem given in Lemma \ref{lem:Ds}
 and we handle here this complexity by the use of different {\it semi-normal forms},
 constructed for the action of 
the pseudo--group $\mathcal{G}$ of local diffeomorphisms 
$\varphi$ such that $\varphi(0)=0$, $\varphi * Y=Y$ and 
feedback  transformation $u\rightarrow-u$ (so that $\sigma_+$ and $\sigma_-$ 
can be exchanged).

\subsection{General concepts and notation}

We consider a local neighborhood $U$ of $q_0\in N$.
If the optimal control $u^*(v)\in [-1,1]$ exists and is unique, it is regular
on an open subset of $U$, union of $U_+$ where $u^*(v)=1$ and $U_-$ where 
$u^*(v)=-1$. 
The surface $S$ which separates $U_+$ from $U_-$
is subanalytic and can be stratified into
\begin{itemize}
  \item a switching locus $W$: closure of the set of points where $u^*$
    is regular and not continuous. 
    We denote by $W_+$ (resp. $W_-$) the points of $W$ where the optimal
    control is $+1$ (resp. $-1$) on $N$,
  \item a cut-locus $C$: closure of the set of points where a trajectory 
    loses its optimality,
  \item singular locus $\Gamma_s$: union of optimal singular trajectories.
\end{itemize}
and these strata can be approximated by semialgebraic sets using 
semi-normal forms.

\subsection{Local syntheses in the exceptional cases}

Take a terminal point $q_1$ of $N$, which can be identified to $0$. 
One wants to describe the time minimal syntheses in a small neighborhood
$U$ of $0$. 
We denote respectively by $\sigma_\pm^0$, $\sigma_s^0$ bang and singular arcs 
terminating at $0$ and we consider only the exceptional case where the arc is 
tangent to $N$, which splits into the bang exceptional case or the 
singular exceptional case.
The syntheses are described in details in \cite{bonnard1997,launay1997} up to
the codimension two situations and we recall the main points to be applied
to the McKeithan network.

\subsubsection{The bang exceptional case}

The neighborhood $U$ of $0$ can be split into two domains denoted by
$U_+$ on which the optimal control is $u=+1$ and the $U_-$ where it is given 
by $u=-1$.
We have to consider the two cases.\\

\paragraph{\bf Generic case (codimension one)}

In this case both arcs $\sigma_+^0$ and $\sigma_-^0$ arc tangent to $N$ but with 
a contact of order $2$.
Using the concept of unfolding, one can define a $C^0$-foliation of $U$ by invariant 
planes so that in each plane the system takes the semi-normal form:
\begin{equation}
  \begin{aligned}
    &\dot x = by + o(|x|,|y|)
    \\
    &\dot y = X_2(x,y)+u, \qquad |u|\le 1,
  \end{aligned}
  \label{mod:codim1}
\end{equation}
where $b=n\cdot [Y,X]\neq 0$, which can be normalized to $1$, $n=(1,0)$ being
the normal to $N$ identified to $x=0$.
Moreover one can assume that $1+X_2(0)>0$ and we have two cases.
\begin{proposition}
  Using the previous normalizations we have two cases described in 
  Fig.\ref{fig:generic_case}.
  \label{prop:generic_case}
\end{proposition}

\begin{figure}
  \centering
  \begin{tabular}{cccc}
    \def\svgwidth{0.12\textwidth}
\begingroup%
  \makeatletter%
  \providecommand\color[2][]{%
    \errmessage{(Inkscape) Color is used for the text in Inkscape, but the package 'color.sty' is not loaded}%
    \renewcommand\color[2][]{}%
  }%
  \providecommand\transparent[1]{%
    \errmessage{(Inkscape) Transparency is used (non-zero) for the text in Inkscape, but the package 'transparent.sty' is not loaded}%
    \renewcommand\transparent[1]{}%
  }%
  \providecommand\rotatebox[2]{#2}%
  \newcommand*\fsize{\dimexpr\f@size pt\relax}%
  \newcommand*\lineheight[1]{\fontsize{\fsize}{#1\fsize}\selectfont}%
  \ifx\svgwidth\undefined%
    \setlength{\unitlength}{222.43803406bp}%
    \ifx\svgscale\undefined%
      \relax%
    \else%
      \setlength{\unitlength}{\unitlength * \real{\svgscale}}%
    \fi%
  \else%
    \setlength{\unitlength}{\svgwidth}%
  \fi%
  \global\let\svgwidth\undefined%
  \global\let\svgscale\undefined%
  \makeatother%
  \begin{picture}(1,1.67127755)%
    \lineheight{1}%
    \setlength\tabcolsep{0pt}%
    \put(0,0){\includegraphics[width=\unitlength,page=1]{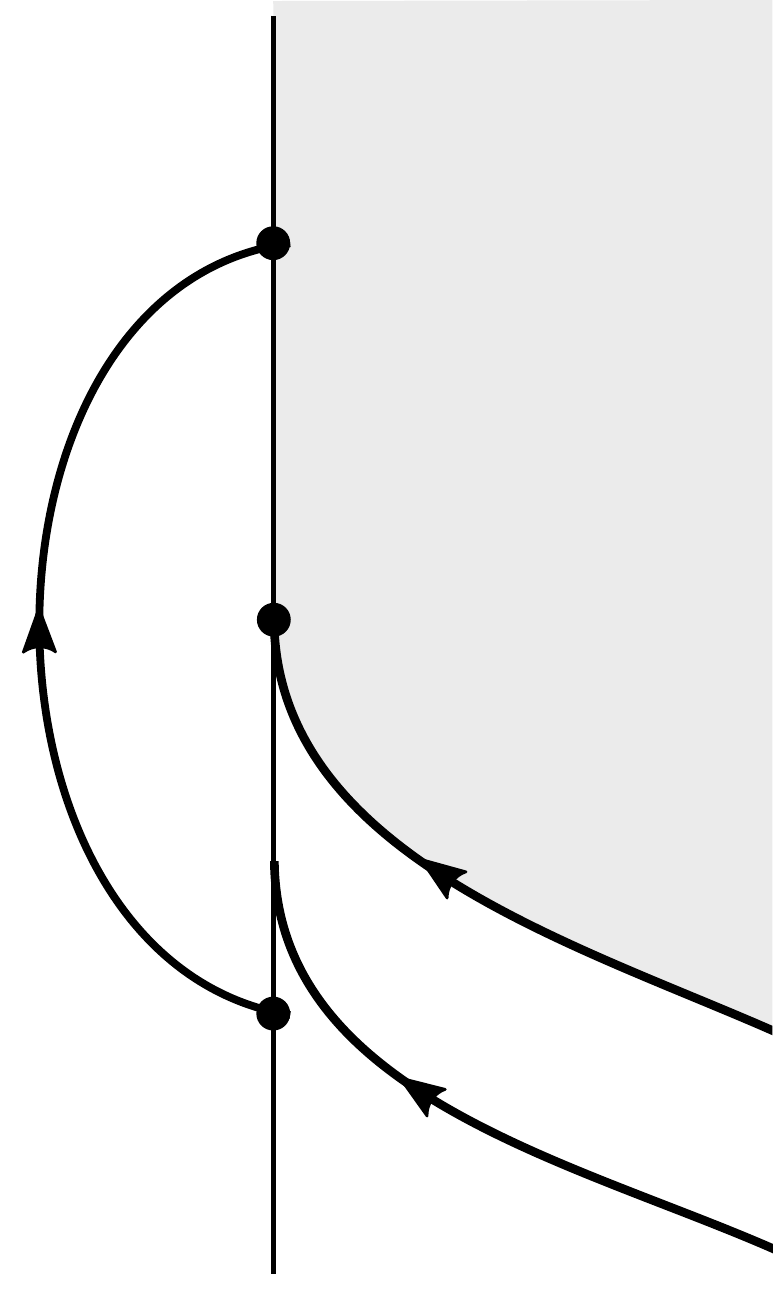}}%
    \put(0.11657536,1.53961183){\color[rgb]{0,0,0}\makebox(0,0)[lt]{\lineheight{1.25}\smash{\begin{tabular}[t]{l}$N$\end{tabular}}}}%
    \put(-0.25214551,0.80442393){\color[rgb]{0,0,0}\makebox(0,0)[lt]{\lineheight{1.25}\smash{\begin{tabular}[t]{l}$\sigma_+$\end{tabular}}}}%
    \put(0.60987423,0.55160235){\color[rgb]{0,0,0}\makebox(0,0)[lt]{\lineheight{1.25}\smash{\begin{tabular}[t]{l}$\sigma_-$\end{tabular}}}}%
    \put(0.58088481,0.21118766){\color[rgb]{0,0,0}\makebox(0,0)[lt]{\lineheight{1.25}\smash{\begin{tabular}[t]{l}$\sigma_-$\end{tabular}}}}%
    \put(0.38234673,0.8344701){\color[rgb]{0,0,0}\makebox(0,0)[lt]{\lineheight{1.25}\smash{\begin{tabular}[t]{l}$0$\end{tabular}}}}%
  \end{picture}%
\endgroup%
 & & &
    \def\svgwidth{0.1\textwidth}
\begingroup%
  \makeatletter%
  \providecommand\color[2][]{%
    \errmessage{(Inkscape) Color is used for the text in Inkscape, but the package 'color.sty' is not loaded}%
    \renewcommand\color[2][]{}%
  }%
  \providecommand\transparent[1]{%
    \errmessage{(Inkscape) Transparency is used (non-zero) for the text in Inkscape, but the package 'transparent.sty' is not loaded}%
    \renewcommand\transparent[1]{}%
  }%
  \providecommand\rotatebox[2]{#2}%
  \newcommand*\fsize{\dimexpr\f@size pt\relax}%
  \newcommand*\lineheight[1]{\fontsize{\fsize}{#1\fsize}\selectfont}%
  \ifx\svgwidth\undefined%
    \setlength{\unitlength}{156.44738866bp}%
    \ifx\svgscale\undefined%
      \relax%
    \else%
      \setlength{\unitlength}{\unitlength * \real{\svgscale}}%
    \fi%
  \else%
    \setlength{\unitlength}{\svgwidth}%
  \fi%
  \global\let\svgwidth\undefined%
  \global\let\svgscale\undefined%
  \makeatother%
  \begin{picture}(1,1.65133143)%
    \lineheight{1}%
    \setlength\tabcolsep{0pt}%
    \put(0,0){\includegraphics[width=\unitlength,page=1]{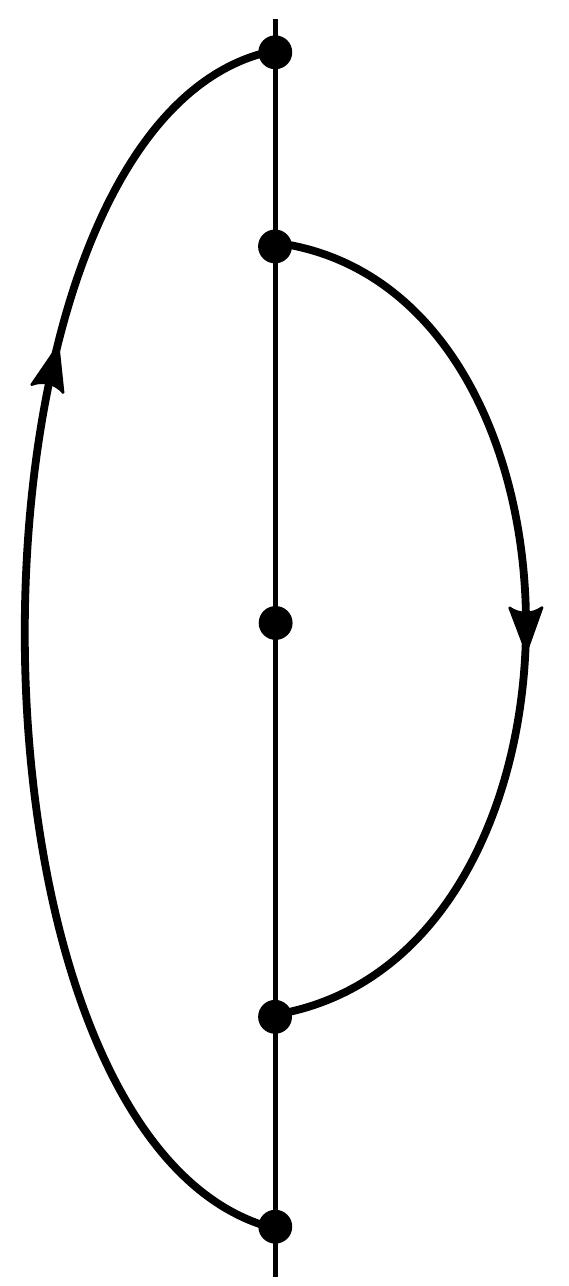}}%
    \put(0.26513544,0.9871651){\color[rgb]{0,0,0}\makebox(0,0)[lt]{\lineheight{1.25}\smash{\begin{tabular}[t]{l}{\small $0$}\end{tabular}}}}%
    \put(1.0017258,0.82679789){\color[rgb]{0,0,0}\makebox(0,0)[lt]{\lineheight{1.25}\smash{\begin{tabular}[t]{l}{\small $\sigma_-$}\end{tabular}}}}%
    \put(-0.3017258,1.62679789){\color[rgb]{0,0,0}\makebox(0,0)[lt]{\lineheight{1.25}\smash{\begin{tabular}[t]{l}{\small $\sigma_+$}\end{tabular}}}}%
    \put(0.50605313,1.93680358){\color[rgb]{0,0,0}\makebox(0,0)[lt]{\lineheight{1.25}\smash{\begin{tabular}[t]{l}$N$\end{tabular}}}}%
  \end{picture}%
\endgroup%

    \\
    Case 1: $X_2(0)>1$ & &  & Case 2: 	$X_2(0)<1$
  \end{tabular}
  \caption{Local synthesis near $\mathcal{E}$ in the generic case.\label{fig:generic_case}}
\end{figure}
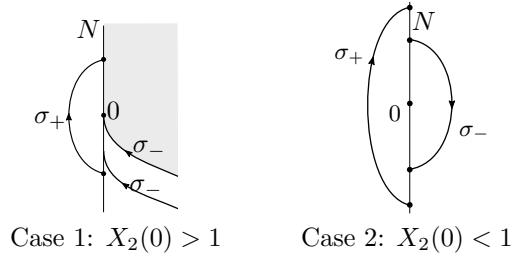

The difference between the two cases is related to different accessibility 
properties of the system.
In the case $X_2(0)>1$, the target $N$ is not accessible from the points in $x>0$
above the arc $\sigma_-^0$ terminating at $0$.
In the case $X_2(0)<1$, each point of $U$ can be steered in minimum time 
to $N$, the domain $U_+$ where the optimal control is $+1$ being $x<0$ and the 
domain $U_-$ with optimal  control $-1$ being $x>0$.\\

\paragraph{\bf Codimension two case}
A more complex situation occurs assuming that the arc $\sigma_-^0$  has a contact
of order three with $N$ while $\sigma_+^0$ has a contact of order two.
The optimal syntheses cannot be described by foliating $N$ by 
$2d$-planes as in the previous cases.

One needs to introduce the following assumptions.
We assume that $Y=\frac{\partial}{\partial z}$, $N$ is the plane $x=0$ 
parameterized as the image of: $(w,s)\mapsto (0,w,s)$.
Denoting $n=(1,0,0)$, the normal to $N$ at $0$, we assume
\begin{itemize}
  \item bang exceptional case: $n\cdot X(0)=0$,
    $n\cdot [Y,X]\neq 0$,
  \item $\det(X,Y,[Y,X])\neq 0$ at $0$,
  \item $\{n\cdot X=0\} \cap N$ is a curve which is neither tangent to $X$ nor to
    $Y$ at $0$.
\end{itemize}

We introduce the following normalization: along the $y$-axis, $n\cdot X=0$
and $[X,Y]=\frac{\partial}{\partial x}$.

Using the concept of semi-normal form the optimal syntheses can be described by
the following model:
\begin{equation}
  \left\{
    \begin{array}{ll}
      \dot x = z
      \\
      \dot y =b, & N:(w,s)\mapsto (0,w,s)
      \\
      \dot z = 1+u+y
    \end{array},
  \right.
  \label{mod:codim2}
\end{equation}
and we have two types of time minimal syntheses.

\begin{proposition}
  Assume $b<0$. Then each point of $U$ can be steered to $N$.
  Moreover
  \begin{enumerate}
    \item $U^+\setminus N \subset \{x<0\}$ 
      and
      $U^-\setminus N \subset \{x>0\}$
    \item Optimal trajectories $\sigma_-$ arrive at any point 
      $(0,w,s<0)$ or $(0,w\ge 0,s)$ of $N\cap U$.
  \end{enumerate}
  The optimal synthesis is described by Fig.\ref{fig:codim2bn}.
  \label{prop:codim2bn}
\end{proposition}

\begin{figure}
  \centering
  \def\svgwidth{0.78\textwidth}
  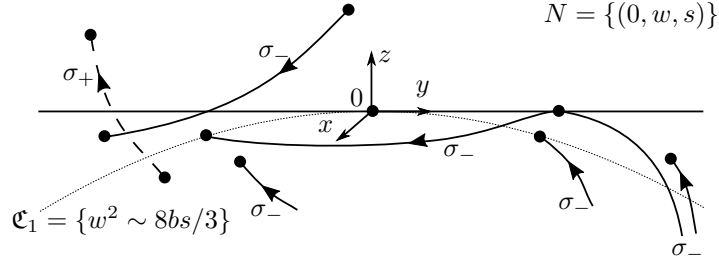
  \caption{Local synthesis near $\mathcal{E}$ in the codimension $2$ 
    case for $b<0$. 
    The dashed curves are in the region $x\le <0$.
  \label{fig:codim2bn}}
\end{figure}

\begin{proposition}
  Assume $b>0$. In this case the system is not locally controllable at $0$.
  We represent on Fig.\ref{fig:codim2bp} the synthesis in this case.
  \label{prop:codim2bp}
\end{proposition}

\begin{figure}
  \centering
  \def\svgwidth{0.78\textwidth}
  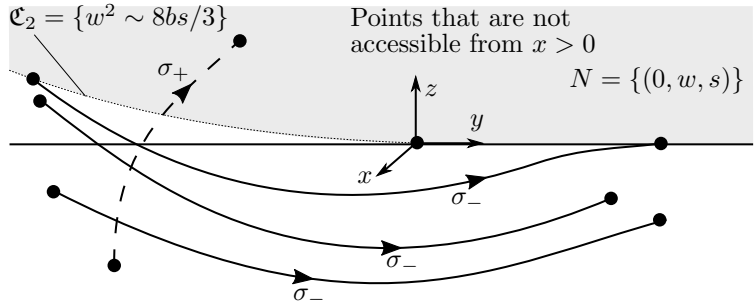
  \caption{Local synthesis near $\mathcal{E}$ in the codimension 2 case for $b>0$.\label{fig:codim2bp}}
\end{figure}

We shall refer to \cite{launay1997} for the full details of the 
computation and description of the syntheses.

\subsubsection{The singular exceptional case}

In this case we can assume $Y=\frac{\partial}{\partial z}$, $N$ is the plane $x=0$,
the normal to $N$ at $0$ is $n=(1,0,0)$ and moreover:

\paragraph{\it Singular exceptional case}
\begin{equation*}
  \left\{
    \begin{array}{l}
      n\cdot X(0)=0
      \\
      n\cdot [Y,X](0)=0
    \end{array}
  \right.
\end{equation*}
and we add the following generic conditions 
\begin{itemize}
  \item    $X \text{ and } Y \text{ are independent at } 0$,
  \item   $\det(Y,[Y,X],\text{ad}^2Y\cdot X)\neq 0 \text{ at } 0$,
  \item  $\{n\cdot X(0)=0\} \cap N \text{ is a curve which is not tangent to } X 
    \text{ at } 0$.
\end{itemize}

Using the concept of semi-normal form the optimal syntheses can be described by the
following 
\begin{equation}
  \left\{
    \begin{array}{l}
      \dot x = y + z^2
      \\
      \dot y =b + b_1\, z
      \\
      \dot z = c+u
    \end{array}.
  \right.
  \label{mod:codim2sing}
\end{equation}
The different syntheses are described in \cite{launay1997} and we present
hereafter a method of computing the time-minimal synthesis for
\eqref{mod:codim2sing} using symbolic computations.

In this model, the exceptional locus $\mathcal{E}\cap N$ is approximated by the
parabola: $w+s^2=0$ and we denote by $\mathcal{E}_-: \{q\in N, \ n\cdot X(q)<0\}$
and $\mathcal{E}_+ =\{q\in N,\ n\cdot X(q)>0\}$.

We have six cases that we can classify using the model.
\begin{itemize}
  \item{\it Case 1:} $b>0,\ u_s(0)>3$.
  \item{\it Case 2:} $b>0,\ 1<u_s(0)<3$.
  \item{\it Case 3:} $b>0,\ 0<u_s(0)<1$.
  \item{\it Case 4:} $b<0,\ u_s(0)>3$.
  \item{\it Case 5:} $b<0,\ 1<u_s(0)<3$.
  \item{\it Case 6:} $b<0,\ 0\le u_s(0)< 1$.
\end{itemize}

This describes the complete classification under generic assumptions 
and the stratification of $S$ can be computed in the original coordinates
and applied to the McKeithan network.
We illustrate the method in the Case $3$ correcting the results obtained
in \cite{launay1997}.

\paragraph*{\bf Illustration of the method based on symbolic computation.}

We present an algorithm to compute an approximation of the surface $S$
in the codimension $2$ case, more specifically we treat the Case 3 
described above and given by the model \eqref{mod:codim2sing}
with $b>0$ and such that the singular trajectory arriving at $0$ is not saturating,
that is $0 \le b_1/2-c < 1$.

\paragraph*{\it Method.}
The following steps involve symbolic computation to obtain the optimal
policy based on \cite{kupka1987}.

1. \ 
Take $q(0)=(0,w,s)\in N\cap U$.
We first determine the stratification of the surface $N$.
Since $Y$ is tangent to $N$, $q(0)$ is a switching point.
If $q(0)$ is an ordinary switching point, the optimal control is regular:
$u(0)=\text{sign}(p(0)\cdot [Y,X](0))$.
If it is a fold point, the optimal control may be singular and the optimal
policy is determined using \cite{bonnard1993}. 
Note that since $u_s(0)=b_1/2-c<1$, 
the singular trajectories are admissible and are either hyperbolic
or elliptic,
which corresponds respectively to time minimizing or time maximizing trajectories.

2.\ 
We then integrate the system backward in time from $q(0)$ 
and compute the equations characterizing the switching surface, the splitting locus $C_s$ and the singular locus $\Gamma_s$.

A trajectory 
$\sigma_{\varepsilon}, \varepsilon\in \{-1,1\}$ 
can switch at time $t_1^{\varepsilon}<0$,
can intersect the surface $N$ at time $t_2^{\varepsilon}<0$ or 
there may exist a time $t_3<0$ and a point $q_3\in V$ 
such that $\exp(-t_3(X+Y))(q_3)=q(0)$, 
$\exp(-t_3(X-Y))(q_3)\in N$.

The weights of the variable $t,s,w$ is respectively $1,2,1$.
We develop the regular flow using Taylor expansion  up to order
$3$ in $t$, we obtain :
\begin{equation}
  \begin{array}{llcl}
    &p_3(t_1^{\varepsilon}) = 0 &\Rightarrow 
    &t_1^{\varepsilon} = \frac{2s}{b_1/2-c-\varepsilon} + \dots,
    \\
    &e^{t_2^{\varepsilon}(X+\varepsilon Y)}(q(0)) \in N
    &\Rightarrow 
    &t_2^{\varepsilon} = \frac{-2}{b} \left(w+s^2  \right) + \dots,
    \\
    & \exists (t_3,q)\in \mathbb{R}_-^*\times U, \ \left\{
      \begin{array}{l}
        e^{-t_3(X+Y)}(q)\in N\\
        e^{-t_3(X-Y)}(q)\in N
      \end{array}
    \right.
	& \Rightarrow
    &t_3= \frac{3s}{b_1/2-c-3} + \dots
  \end{array}
  \label{eq:tis}
\end{equation}
and a parameterization of 
\begin{itemize}
  \item the switching surface $W_-=(x(w,s),y(w,s),z(w,s))$ is 
\end{itemize}
\begin{small}
  \begin{equation}
    \begin{aligned}
      &x(w,s) = \frac{s^2 \left(8 b_1^2+\nu_1^2\right)+3 w \nu_1^2+6 b s \nu_1}{3\nu_1^3/4s} 
       + \dots,
      \\
      &y(w,s) = \frac{4 s \left(b \nu_1 +b_1^2 s\right)}{\nu_1^2}+w +\dots ,
      \\
      &z(w,s) = \frac{s (b_1+2 (c+\varepsilon))}{\nu_1} +\dots,
    \end{aligned}
  \end{equation}
  where $\nu_1 = b_1-2 (c+\varepsilon)$.
  \end{small}
\begin{itemize}
  \item the singular surface $\Gamma_s\coloneqq \Gamma_s(t,w)$ is
    \begin{equation}
      \Gamma_s = \left(\frac{b t^2}{2}+\frac{b_1^2 t^3}{6}+t w,b
      t+\frac{b_1^2 t^2}{4}+w,\frac{b_1 t}{2}\right) +\dots,
    \end{equation}
  \item and the splitting locus $C_s=(x(w,s),y(w,s),z(w,s))$ is
\end{itemize}
  \begin{equation}
    \begin{aligned}
      &x(w,s) =\frac{2
        s^2 \left(b_1 c+b_1 (2 b_1-9)+2c^2+6\right)+w \nu_2^2}{\nu_2^3/6s} +\frac{3 b s \nu_2}{\nu_2^^3/6s} +\dots,
      \\
      &y(w,s) =\frac{6 s (b \nu_2+b_1 s
      (b_1+c-3))}{\nu_2^2}+w+\dots,
      \\
      &z(w,s) =\frac{s(b_1+4 c)}{\nu_2}+\dots.
    \end{aligned}
  \end{equation}
  where $\nu_2=b_1-2 c-6$.

3. \
The optimal policy is deduced by computing $t^*=\max(t_1^{\varepsilon},t_2^{\varepsilon},t_3)$ with
$\varepsilon=\text{sign}(s\, (n\cdot X(0,w,s)))$
and we represent
in Fig.\ref{fig:amax} the region of $N$ where $t^*(w,s)$ corresponds to a switching time, a splitting time or a time at which
the trajectory intersects $N$.
The surface $S$ separating $U_+$ and $U_-$ is the union of the switching 
surface $W_-$, the singular surface $\Gamma_s$ foliated by singular arcs,
the splitting locus $C_s$ and a subset 
$\tilde{\mathcal{E}}_- \subset \mathcal{E}_-$ from which $\sigma_{\pm}$
intersects $N$ in the green region of Fig. \ref{fig:amax}.
The set $S$ is represented in Fig. \ref{fig:strata-a} together with 
some trajectories emanating from $N$ to the components of $S$.
The green trajectories starting from $\mathcal{E}_+$ intersect $N$ in 
$\tilde{\mathcal{E}}_-$.    

\begin{figure}[htpb]
  \centering
  \includegraphics[scale=0.25]{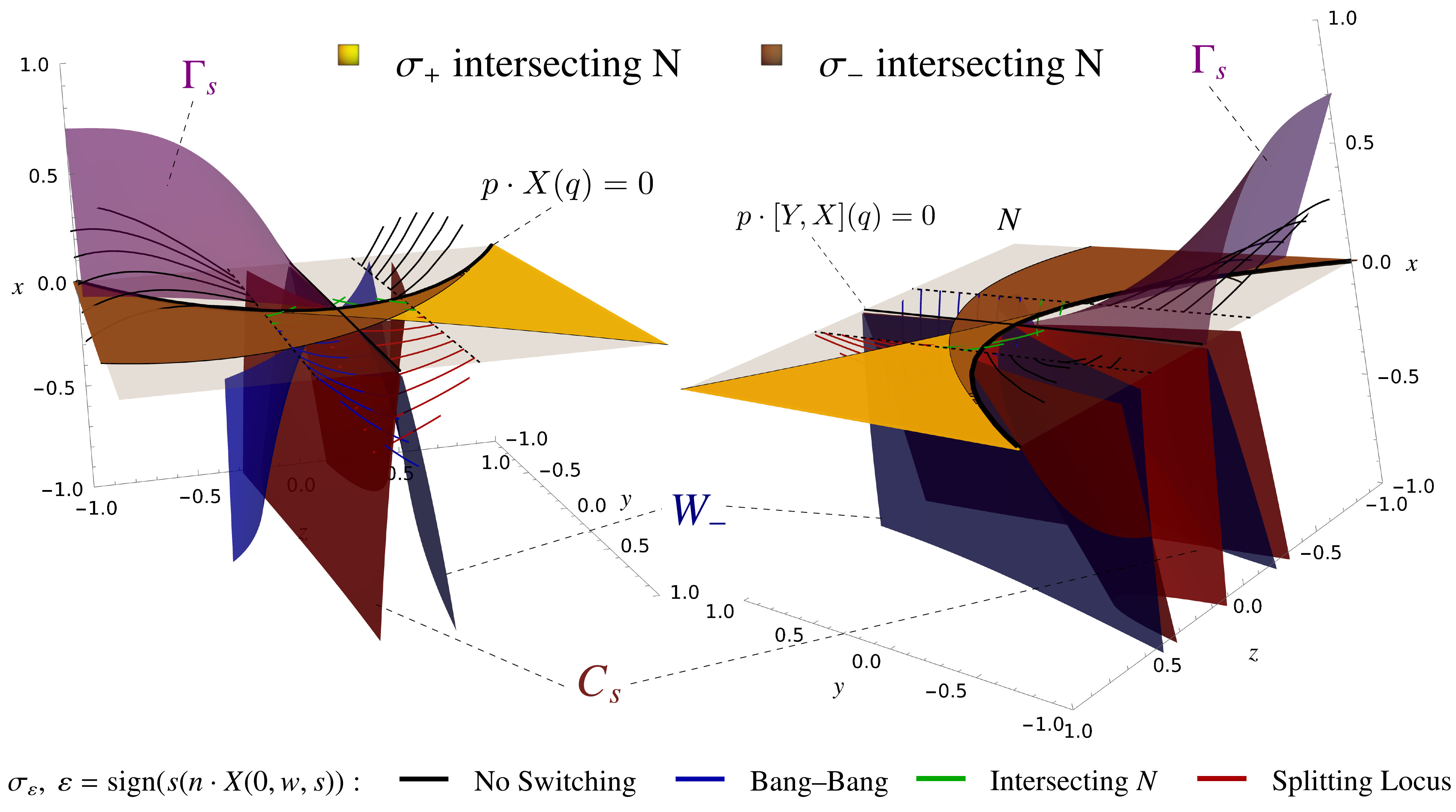}
  \caption{Strata of the surface $S$ separating the regions of $U$ where the control is $\pm 1$ for the model \eqref{mod:codim2sing} with
    $b=b_1=1,\ c=0$.
    We also represent the regions where $\sigma_\pm$ intersect $N$ for $t<0$
    and several trajectories which generate a switching locus and
    a splitting locus.   
  \label{fig:strata-a}}
\end{figure}

\begin{figure}[htpb]
  \centering
  \includegraphics[scale=0.5]{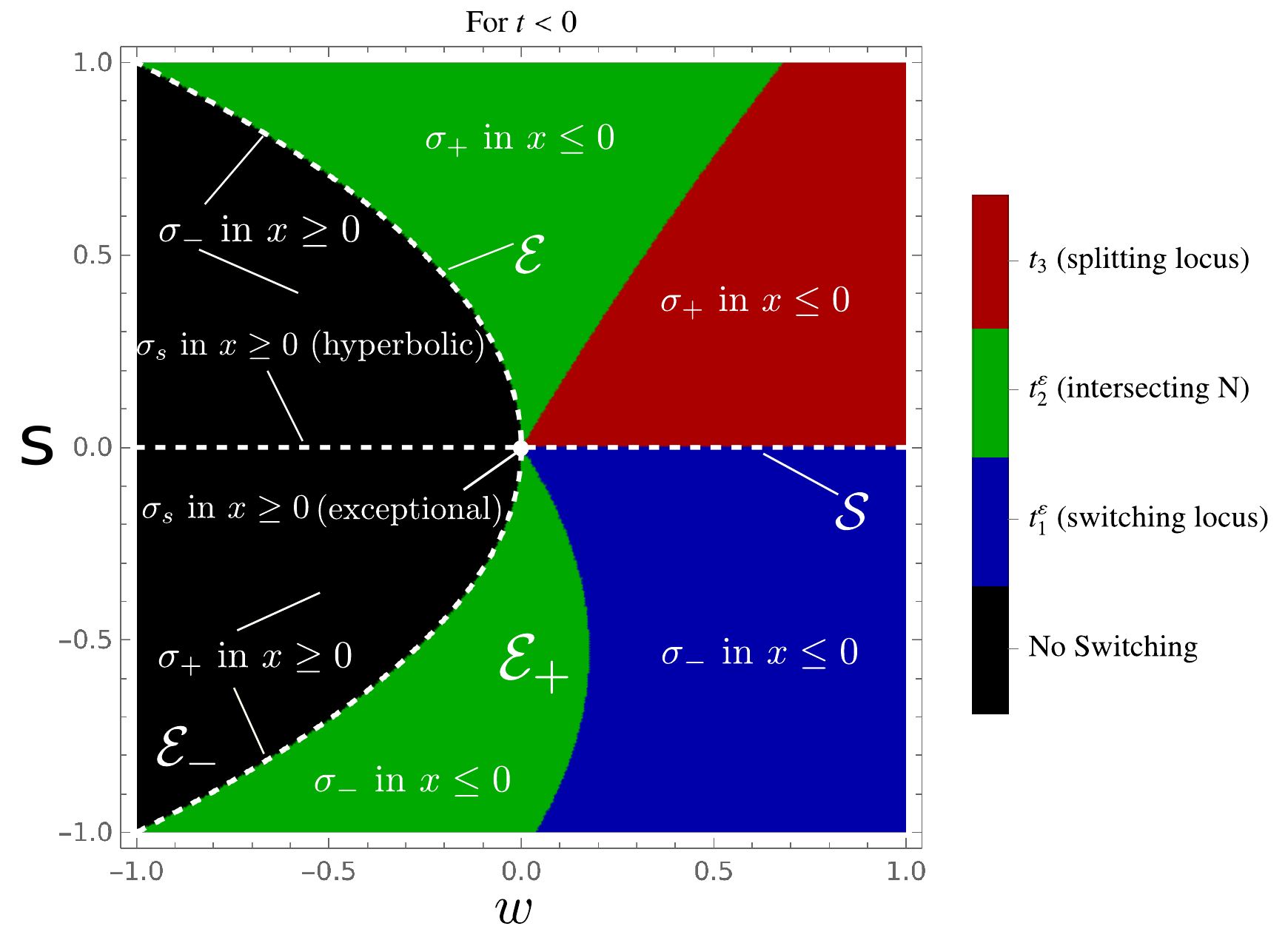}
  \caption{Minimum time 
    $t^*=\max(t_1^{\varepsilon},t_2^{\varepsilon},t_3)$ to reach $(0,w,s)\in N$ from a neighbourhood $U$ of $0$
    for the model \eqref{mod:codim2sing} with
    $b=b_1=1,\ =c=0$. The exceptional locus is $\mathcal{E} : y=-z^2$ and the singular locus
    is $\mathcal{S}: n\cdot [Y,X](q)=0 : z=0$.  
    \label{fig:amax}
  }
\end{figure}

\begin{figure}[htpb]
  \centering
  \def\svgwidth{0.47\textwidth}
  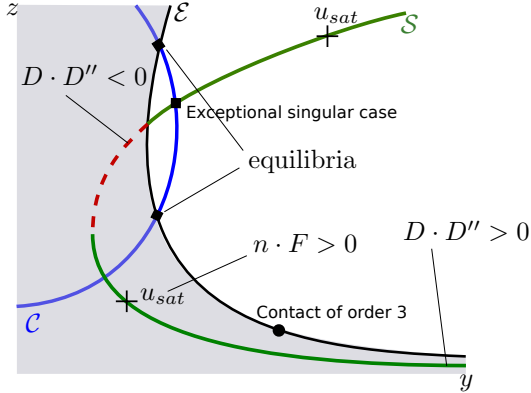
  \caption{
    Traces of the surfaces $\mathcal{E}$, $\mathcal{C}$
    and $\mathcal{S}$ on the terminal manifold $N$ that represent
  its stratification.}
  \label{fig:stratification_mck}
\end{figure}


\section{Conclusion}

In this article, we provide a general framework to analyze accessibility 
properties of abnormal geodesics using two case studies.

The first case is motivated by the historical Zermelo navigation problem,
which is generalized into $2d$-navigation problem.
A barrier is formed by decomposing the state space into strong and weak current
domains.
Abnormal geodesics are limit curves of the set of admissible directions.
We show that they reflect on this barrier with a cusp singularity and we analyze 
this phenomenon using a semi-normal form to evaluate the value function.

The second case is motivated by the problem of optimizing the production of one
species for chemical reactors using the derivative of the temperature as control.
In this case, the maximized Hamiltonian is nonsmooth and geodesics are 
concatenation of bang and singular arcs.
Again we concentrate to the case of abnormal cases.
The various cases are described using semi-normal forms to compute the time
minimal synthesis and evaluate the value function.

\medskip
Received xxxx 20xx; revised xxxx 20xx; early access xxxx 20xx.
\medskip

\end{document}